\documentclass[11pt, reqno, psamsfonts]{amsart}
\pdfoutput=1

\usepackage{amssymb}
\usepackage{amsthm}
\usepackage{amsmath}
\usepackage{latexsym}
\usepackage[T1]{fontenc}
\usepackage[utf8]{inputenc}
\usepackage[russian, french, english]{babel}

\usepackage{graphicx}
\usepackage{wrapfig}
\usepackage{mathtools}
\usepackage{tikz}
\usepackage{amsbsy}
\usepackage[inline]{enumitem}
\usepackage{mathrsfs}
\usetikzlibrary{shapes,snakes}
\usetikzlibrary{arrows.meta}
\usetikzlibrary{decorations.pathmorphing}
\usetikzlibrary{patterns}
\usetikzlibrary{calc}
\usepackage{float}
\usepackage{array}
\usepackage{subcaption}
\usepackage{multicol}
\usepackage{stmaryrd}
\usepackage{cancel} %to strike through text
\usepackage{algorithm}
\usepackage{algpseudocode}
\makeatletter
\algnewcommand{\LineComment}[1]{\Statex \hskip\ALG@thistlm {\color{gray}\texttt{// #1}}}
\makeatother

\counterwithin{algorithm}{section}
\usepackage[justification=centering, labelfont=bf]{caption}
\usepackage{lmodern}
\usepackage{mathabx}

\usepackage{upgreek}
\usepackage{titlesec}
\usepackage{bbm}
\usepackage[spacing=true,kerning=true,babel=true,tracking=true]{microtype}
\usepackage[shortcuts]{extdash}
\usepackage[foot]{amsaddr}
\usepackage[left=1in,right=1in,top=1in,bottom=1in,bindingoffset=0cm]{geometry}
\usepackage{bm}
\usepackage{centernot}
\usepackage{mdframed}
\usepackage{hyperref}
\usepackage{mathdots}
\usepackage{xspace}
\hypersetup{
    colorlinks=true,%change to true to get colors
    linkcolor=blue,
    filecolor=magenta,      
    urlcolor=red,
    citecolor=gray,
}
\allowdisplaybreaks

\usepackage[backend=biber, style=alphabetic, sorting=nyt, maxnames=100,backref=true, firstinits=true]{biblatex}
\title{Bounds for the independence and chromatic numbers of locally sparse graphs}
\date{}
\author{\lsstyle Abhishek~Dhawan}
\thanks{This research was partially supported by the Georgia Tech ARC-ACO Fellowship, NSF grant DMS-2053333 (PI: Cheng Mao), the NSF CAREER grant DMS-2239187 (PI: Anton Bernshteyn), and the NSF RTG grant DMS-1937241.}
\email{adhawan2@illinois.edu}
\address{\textls{\normalfont{}Department of Mathematics, University of Illinois Urbana-Champaign, Urbana, IL, USA}}

\newtheoremstyle{bfnote}%
{}{}%
{\slshape}{}%
{\bfseries}{\bfseries.}%
{ }%
{\thmname{#1}\thmnumber{ #2}\thmnote{ \ep{\normalfont{}#3}}}

\theoremstyle{bfnote}
\newtheorem{theo}{Theorem}[section]
\newtheorem*{theo*}{Theorem}
\newtheorem{prop}[theo]{Proposition}
\newtheorem*{prop*}{Proposition}
\newtheorem{Lemma}[theo]{Lemma}
\newtheorem{claim}{Claim}[theo]
\newtheorem{corl}[theo]{Corollary}
\newtheorem{conj}[theo]{Conjecture}
\newtheorem{fact}[theo]{Fact}

\newtheorem*{corl*}{Corollary}

\newcounter{ForClaims}[section]

\theoremstyle{definition}
\newtheorem{defn}[theo]{Definition}
\newtheorem*{defn*}{Definition}

\newtheorem*{exmp*}{Example}

\theoremstyle{remark}
\newtheorem*{ques*}{Question}
\newtheorem*{remk*}{Remark}
\newtheorem*{obs*}{Observation}

\newcommand*{\myproofname}{Proof}
\newenvironment{claimproof}[1][\myproofname]{\begin{proof}[#1]}{\end{proof}}

\newcommand{\set}[1]{\{#1\}}
\newcommand{\N}{{\mathbb{N}}}

\newcommand{\R}{\mathbb{R}}

\renewcommand{\P}{\mathbb{P}}
\newcommand{\E}{\mathbb{E}}
\renewcommand{\epsilon}{\varepsilon}
\newcommand{\eps}{\epsilon}
\renewcommand{\phi}{\varphi}
\renewcommand{\theta}{\vartheta}
\renewcommand{\leq}{\leqslant}
\renewcommand{\geq}{\geqslant}
\newcommand{\defeq}{\coloneqq}
\newcommand{\im}{\mathrm{im}}
\newcommand{\bemph}[1]{{\normalfont#1}} % define how emphasised brackets should look
\newcommand{\ep}[1]{\bemph{(}#1\bemph{)}} % parentheses

\newcommand{\emphdef}[1]{\textbf{\textit{{#1}}}}

\newcommand{\emphd}[1]{\emphdef{#1}}

\newcommand{\LLL}{\text{Lov\'asz Local Lemma}}
\numberwithin{equation}{section}

\newcommand{\boldk}{\mathbf{k}}
\newcommand{\boldr}{\mathbf{r}}
\newcommand{\boldeps}{\bm{\eps}}

\titleformat{\subsection}[block]{\bfseries}{\thesubsection.}{1ex}{}
\titleformat{\subsubsection}[runin]{\itshape}{\bfseries\upshape\thesubsubsection.}{1ex}{}[.---]

\titleformat{\section}[block]{\scshape\filcenter}{\thesection.}{1ex}{}
\titlespacing*{\section}{0pt}{*3}{*1}
\titlespacing*{\subsection}{0pt}{*3}{*1}
\titlespacing*{\subsubsection}{0pt}{*1.5}{*0}

\renewbibmacro{in:}{}

\renewbibmacro*{volume+number+eid}{%
	\printfield{volume}%
	\setunit*{\addnbspace}% NEW (optional); there's also \addnbthinspace
	\printfield{number}%
	\setunit{\addcomma\space}%
	\printfield{eid}}
\DeclareFieldFormat[article]{volume}{\textbf{#1}\space}
\DeclareFieldFormat[article]{number}{\mkbibparens{#1}}

\DeclareFieldFormat{journaltitle}{#1,}
\DeclareFieldFormat[thesis]{title}{\mkbibemph{#1}\addperiod}
\DeclareFieldFormat[article, unpublished, thesis]{title}{\mkbibemph{#1},}
\DeclareFieldFormat[book]{title}{\mkbibemph{#1}\addperiod}
\DeclareFieldFormat[unpublished]{howpublished}{#1, }

\DeclareFieldFormat{pages}{#1}

\DeclareFieldFormat[article]{series}{Ser.~#1\addcomma}

\setlist{topsep=3pt,itemsep=3pt}

\begin{document}

\vspace*{0pt}

\maketitle
\begin{abstract}
    In this note we consider a more general version of local sparsity introduced recently by Anderson, Kuchukova, and the author. In particular, we say a graph $G = (V, E)$ is $(k, r)$-locally-sparse if for each vertex $v \in V(G)$, the subgraph induced by its neighborhood contains at most $k$ cliques of size $r$. For $r \geq 3$ and $\epsilon \in [0, 1]$, we show that an $n$-vertex $(\Delta^{\epsilon r}, r)$-locally-sparse graph $G$ of maximum degree $\Delta$ satisfies $\alpha(G) \geq (1-o(1))\dfrac{n}{\eta\Delta}$ and $\chi(G) \leq \Theta\left(\eta\Delta\right)$, where $\eta \defeq \epsilon + \dfrac{r\log\log \Delta}{\log \Delta}$. For $\epsilon$ not too large, the hidden constant in the $\Theta(\cdot)$ can be taken to be $1+o(1)$. Setting $\epsilon = 0$, we recover classical results on $K_{r+1}$-free graphs due to Shearer and Johansson, which were more recently improved by Davies, Kang, Pirot, and Sereni. We prove a stronger result on the independence number in terms of the occupancy fraction in the hard-core model, and establish a local version of the coloring result in the more general setting of correspondence coloring.
\end{abstract}

\vspace{0.3in}

\paragraph{\textbf{Basic Notation.}}
All graphs considered here are finite, undirected, and simple.
For $n \in \N$, we let $[n] \defeq \set{1, \ldots, n}$.
For a graph $G$, its vertex and edge sets are denoted $V(G)$ and $E(G)$ respectively.
We say $G$ is \emphd{complete} if every pair of vertices in $G$ form an edge, and \emphd{edgeless} if $G$ contains no edges.

For a vertex $v \in V(G)$, $N_G(v)$ denotes the neighbors of $v$ and $\deg_G(v) \defeq |N_G(v)|$ denotes the degree of $v$.
We let $\Delta(G)$, $\delta(G)$, and $d(G)$ denote the maximum, minimum, and average degrees of $G$, respectively.
Let $N_G[v] \defeq N_G(v) \cup \set{v}$ denote the closed neighborhood of $v$.
For a subset $U \subseteq V(G)$, the subgraph induced by $U$ is denoted by $G[U]$, and $N_G(U)$ is the set of all vertices adjacent to a vertex in $U$, i.e., $N_G(U) = \bigcup_{u \in U}N_G(u)$.
We drop the subscript $G$ when the context is clear.

A set $I \subseteq V(G)$ is \emphd{independent} if $G[I]$ is edgeless, and a set $K \subseteq V(G)$ forms a \emphd{clique} if $G[K]$ is complete.
A proper coloring of $G$ is a function $\phi\,:\,V(G) \to \N$ such that no edge $uv \in E(G)$ satisfies $\phi(u) = \phi(v)$.
The \emphd{independence number} of $G$ (denoted $\alpha(G)$) is the size of the largest independent set in $G$, the \emphd{clique number} (denoted $\omega(G)$) is the size of the largest clique in $G$, and the \emphd{chromatic number} of $G$ (denoted $\chi(G)$) is the minimum value $q$ such that $G$ admits a proper coloring $\phi$ where the image of $\phi$ satisfies $|\im(\phi)| = q$.

\section{Introduction}\label{sec:intro}

\subsection{Background and Results}\label{subsec:background}

It is well-known that a graph $G$ satisfies $\alpha(G) \geq \frac{n}{\Delta(G)+1}$ and $\chi(G) \leq \Delta(G) + 1$.
These bounds are often referred to as the \textit{greedy bounds}.
For $\alpha(G)$, one can obtain a stronger bound of $\frac{n}{d(G)+1}$, which we refer to as the \textit{random greedy bound}.
A natural question is the following: under what structural constraints can we obtain improved bounds on $\alpha(\cdot)$ and $\chi(\cdot)$? 
In this note, we consider the constraint of local sparsity.
For $k \in \R$, we say a graph $G$ is \emphd{$k$-locally-sparse} if for each $v \in V(G)$, the subgraph $G[N(v)]$ contains at most $\lfloor k \rfloor$ edges.
Such a notion has been considered extensively (see for example \cite{alon1999coloring, frieze2006randomly, davies2020algorithmic, pirot2021colouring, dhawan2024palette}).
We consider the following generalization of this notion introduced by Anderson, Kuchukova, and the author in \cite{local_sparse}:

\begin{defn}\label{def:local_sparse}
    For graphs $F$ and $G$, a \emphd{copy} of $F$ in $G$ is a subgraph $H \subseteq G$ which is isomorphic to $F$.
    Let $F$ be a graph and let $k \in \R$.
    A graph $G$ is \emphd{$(k,F)$-sparse} if $G$ contains at most $\lfloor k \rfloor$ copies of $F$ (not necessarily vertex-disjoint).
    A graph $G$ is \emphd{$(k,F)$-locally-sparse} if, for every $v \in V(G)$, the induced subgraph $G[N(v)]$ is $(k, F)$-sparse.
\end{defn}

When $F = K_r$, we simply call such graphs \emphd{$(k, r)$-sparse} or \emphd{$(k, r)$-locally-sparse}.
Note that $K_{r+1}$-free graphs are $(0, r)$-locally-sparse.
There has been a lot of work in this regime.
Ajtai, Koml\'os, and Szemer\'edi first considered the independence number of $K_3$-free graphs.
They showed that an $n$-vertex $K_3$-free graph $G$ of average degree $d$ has independence number $\alpha(G) = \Omega\left(n\log d/d\right)$ \cite{ajtai1980note}.
For $r \geq 3$, the same group along with Erd\H{o}s showed that $\alpha(G) = \Omega\left(n\log\left(\log d/r\right)/(rd)\right)$ for $K_{r+1}$-free graphs $G$ \cite{ajtai1981turan}.
Shearer improved the constant factor in the bound for $K_3$-free graphs \cite{shearer1983note} and went on to prove the following celebrated result for $r \geq 3$:

\begin{theo}[\cite{shearer1995independence}]\label{theo:shearer}
    Let $r \in \N$ such that $r \geq 3$, and let $n, d$ be sufficiently large.
    For any $K_{r+1}$-free graph $G$ on $n$ vertices such that $d(G) = d$, we have the following:
    % for some constant $c > 0$:
    \[\alpha(G) = \Omega\left(\dfrac{n\,\log d}{r\,d\,\log\log  d}\right).\]
\end{theo}

Our first result is a generalization of the above. 
In particular, rather than considering $K_{r+1}$-free graphs, we consider the case where $G$ contains ``few'' copies of $K_{r+1}$.
Equivalently, on average each vertex is contained in ``few'' copies of $K_{r+1}$.

\begin{theo}\label{theo:independent_avg_deg}
    There exists $\rho > 0$ such that the following holds for $n,\, d$ sufficiently large.
    Let $\eps,\eta \in [0, 1]$, $r \in \N$, and $k \in \R$ be such that
    \[3\,\leq\, r \leq \frac{\rho\,\log\log d}{\log\log\log d}, \quad k \,=\, \frac{n\,d^{\eps\,r}}{r+1}, \quad \text{and} \quad \eta = \eps + \frac{r\log\log d}{\log d}.\]
    For any $n$-vertex $(k, r+1)$-sparse graph $G$ of average degree $d$, we have
    \[\alpha(G) \geq \left(1 - o(1)\right)\dfrac{n}{9\eta d}.\]
\end{theo}

A few remarks are in order.
First, we note that the value $k$ above implies that on average a vertex $v \in V(G)$ is contained in at most $d(G)^{\eps\,r}$ copies of $K_{r+1}$ for such graphs.
Additionally, for $\eps = O(r\log\log d/\log d)$, we obtain the same growth rate as Theorem~\ref{theo:shearer} with the weaker constraint of $(k, r+1)$-sparsity as opposed to $K_{r+1}$-freeness.
Finally, for $\eps = \Omega(r\log\log d/\log d)$ we improve upon the random greedy bound by a factor of roughly $1/\eps$.

Theorem~\ref{theo:independent_avg_deg} follows as a corollary to the following result for locally sparse graphs (see \S\ref{sec:corl} for a proof of the implication):

\begin{theo}\label{theo:independent}
    There exists $\rho > 0$ such that the following holds for $n,\,\Delta$ sufficiently large.
    Let $\eps,\eta \in [0, 1]$, $r \in \N$, and $k \in \R$ be such that
    \[3\,\leq\, r \leq \frac{\rho\,\log\log \Delta}{\log\log\log \Delta}, \quad k \,=\, \Delta^{\eps\,r}, \quad \text{and} \quad \eta = \eps + \frac{r\log\log \Delta}{\log \Delta}.\]
    For any $n$-vertex $(k, r)$-locally-sparse graph $G$ of maximum degree $\Delta$, we have
    \[\alpha(G) \geq (1-o(1))\dfrac{n}{\eta\Delta}.\]
\end{theo}

We note that $k \in [1, \Delta^r]$. 
In particular, as a vertex can contain at most $\binom{\Delta}{r}$ cliques of size $r$ in its neighborhood, this range covers all possible values of $k$.
We prove Theorem~\ref{theo:independent} as a corollary to a more general result on the expected size of an independent set drawn from the hard-core model on $G$ (see Theorem~\ref{theo:occ_frac}).

When considering colorings, the best-known growth rate on the chromatic number for $K_{r+1}$-free graphs is due to Johansson.
Molloy provided a simpler proof of the result by employing the so-called \textit{entropy compression method} \cite{molloy2019list}.
Bernshteyn simplified the proof further by employing the \textit{lopsided local lemma} and extended the result to \textit{correspondence coloring} (described in \S\ref{subsec:corr}) \cite{bernshteyn2019johansson}.
Davies, Kang, Pirot, and Sereni built upon these ideas to prove the following result, which provides the current best-known bound:

\begin{theo}[\cite{DKPS}]\label{theo:dkps}
    For $r \geq 3$, let $G$ be a $K_{r+1}$-free graph of maximum degree $\Delta$ sufficiently large.
    Then,
    \[\chi(G) \leq (1+o(1))\dfrac{r\,\Delta\,\log\log \Delta}{\log\Delta}.\]
\end{theo}

We note that the result holds in the more general settings of \textit{list and correspondence coloring} described in \S\ref{subsec:corr}.
The proof of Theorem~\ref{theo:dkps} relies on certain auxiliary results in the proof of Theorem~\ref{theo:shearer}.
Adapting their strategy, we prove the following bound on the chromatic number of locally sparse graphs:

\begin{theo}\label{theo:coloring}
    There exists $\rho > 0$ such that the following holds for $\Delta$ sufficiently large.
    Let $\eps,\,\eta \in [0, 1]$, $r \in \N$, and $k \in \R$ be such that
    \[3\,\leq\, r \leq \frac{\rho\,\log\log \Delta}{\log\log\log \Delta}, \quad k \,=\, \Delta^{\eps\,r}, \quad \text{and} \quad \eta = \eps + \frac{r\log\log \Delta}{\log \Delta}.\]
    For any $(k, r)$-locally-sparse graph $G$ of maximum degree $\Delta$, we have
    \[\chi(G) \leq (1+o(1))\eta\Delta\min\left\{2,\,\left(\frac{1 + \eps}{1-\eps}\right)\right\}.\]
\end{theo}

For $\eps = o(r\log\log \Delta/\log\Delta)$, we obtain the same result as Theorem~\ref{theo:dkps} with the weaker assumption of local sparsity.
For larger $\eps$, our results improve upon the greedy bound of $\Delta + 1$ by a factor of roughly $\eps$.
Additionally, our results hold for list and correspondence colorings as well (see Corollary~\ref{corl:correspondence}).

Note that rather than having a strict bound on the clique number of $G$ (as is the case in Theorem~\ref{theo:dkps}), we consider the case that vertices are not contained in ``too many'' small cliques.
To see the versatility of our results, consider a $(k, r)$-locally-sparse graph $G$ of maximum degree $\Delta$ sufficiently large.
It is not too difficult to see that $\omega(G) = O(rk^{1/r})$ by considering the neighborhood of a vertex contained in a maximum clique.
In particular, for $k = \Delta^{\eps r}$, Theorem~\ref{theo:dkps} provides the bound $\chi(G) = O(r\Delta^{1+\eps}\log\log \Delta/\log \Delta)$, which is significantly larger than that provided by Theorem~\ref{theo:coloring}.
A result of Bonamy, Kelly, Nelson, and Postle \cite[Theorem~1.6]{bonamy2022bounding} (further improved by Davies, Kang, Pirot, and Sereni in terms of the constant factor \cite[Theorem~30]{DKPS}) provides a stronger bound of $O\left(\Delta\sqrt{\eps + (\log r/\log \Delta)}\right)$, which is still weaker than our result.

\subsection{List and Correspondence Coloring}\label{subsec:corr}

Introduced independently by Vizing \cite{vizing1976coloring} and Erd\H{o}s, Rubin, and Taylor \cite{erdos1979choosability}, \textit{list coloring} is a generalization of graph coloring in which each vertex is assigned a color from its own predetermined list of colors. 
Formally, $L : V(G) \to 2^{\N}$ is a \textit{list assignment} for $G$, and an \textit{$L$-coloring} of $G$ is a proper coloring of $G$ such that each vertex $v \in V(G)$ receives a color from its list $L(v)$.
% proper coloring $\phi: V(G) \to \N$ such that $\phi(v) \in L(v)$ for each $v \in V(G)$. 
When $|L(v)| \geq q$ for each $v \in V(G)$, where $q \in \N$, we say $L$ is \textit{$q$-fold}. 
The \textit{list-chromatic number} of $G$, denoted $\chi_{\ell}(G)$, is the smallest $q$ such that $G$ has an $L$-coloring for every $q$-fold list assignment $L$ for $G$.

\textit{Correspondence coloring} (also known as DP-coloring) is a generalization of list coloring introduced by Dvo\v{r}\'ak and Postle \cite{dvovrak2018correspondence}. 
Just as in list coloring, each vertex is assigned a list of colors, $L(v)$;
in contrast to list coloring, though, the identifications between the colors in the lists are allowed to vary from edge to edge.
That is, each edge $uv \in E(G)$ is assigned a matching $M_{uv}$ (not necessarily perfect and possibly empty) from $L(u)$ to $L(v)$.
A proper correspondence coloring is a mapping $\phi : V(G) \to \N$ satisfying $\phi(v) \in L(v)$ for each $v \in V(G)$ and $\phi(u)\phi(v) \notin M_{uv}$ for each $uv \in E(G)$.
Formally, we describe correspondence colorings in terms of an auxiliary graph known as a \textit{correspondence cover} of $G$.
The definition below appears in earlier works of the author along with Anderson and Bernshteyn \cite{anderson2022coloring, anderson2023colouring}.

\begin{defn}[Correspondence Cover]\label{def:corr_cov}
    A \emphd{correspondence cover} (also known as a \emphd{DP-cover}) of a graph $G$ is a pair $\mathcal{H} = (L, H)$, where $H$ is a graph and $L \,:\,V(G) \to 2^{V(H)}$ such that:
    \begin{enumerate}[label= \ep{\normalfont DP\arabic*}, leftmargin = \leftmargin + 1\parindent]
        \item The set $\set{L(v)\,:\, v\in V(G)}$ forms a partition of $V(H)$.
        \item\label{dp:list_independent} For each $v \in V(G)$, $L(v)$ is an independent set in $H$.
        \item\label{dp:matching} For each $u, v\in V(G)$, the edges of the induced subgraph $H[L(u) \cup L(v)]$ form a matching; this matching is empty whenever $uv \notin E(G)$.
    \end{enumerate}
\end{defn}

We call the vertices of $H$ \emphd{colors}.
If two colors $c$, $c' \in V(H)$ are adjacent in $H$, we say that they \emphd{correspond} to each other.
An \emphd{$\mathcal{H}$-coloring} is a mapping $\phi \colon V(G) \to V(H)$ such that $\phi(v) \in L(v)$ for all $v \in V(G)$.
An $\mathcal{H}$-coloring $\phi$ is \emphd{proper} if the image of $\phi$ is an independent set in $H$.

A correspondence cover $\mathcal{H} = (L,H)$ is \emphdef{$q$-fold} if $|L(v)| \geq q$ for all $v \in V(G)$. 
The \emphdef{correspondence chromatic number} of $G$, denoted by $\chi_{c}(G)$, is the smallest $q$ such that $G$ admits a proper $\mathcal{H}$-coloring for every $q$-fold correspondence cover $\mathcal{H}$.
As correspondence coloring generalizes list coloring, which in turn generalizes ordinary coloring, we have
\begin{align}\label{eq:relationship}
    \chi(G) \,\leq\, \chi_\ell(G) \,\leq\, \chi_c(G).
\end{align}

Our main coloring result is on \textit{local correspondence colorings} of locally sparse graphs. 
A correspondence cover is \textit{local} if for each vertex $v \in V(G)$, the size of its list $L(v)$ is a function of the local structure of $G$ with respect to $v$.
This notion has gathered much interest in recent years in both vertex coloring \cite{davies2020coloring, DKPS, bonamy2022bounding} and edge coloring \cite{bonamy2020edge, dhawan2023multigraph}.
Before we state our result, we introduce the following definition, which is a local version of Definition~\ref{def:local_sparse} for $F = K_r$.

\begin{defn}\label{def:local_local_sparse}
    Let $G$ be a graph and let $\boldk\,:\, V(G) \to \R$ and $\boldr\,:\, V(G) \to \N$.
    A graph $G$ is \emphd{$(\boldk, \boldr)$-locally-sparse} if, for every $v \in V(G)$, the induced subgraph $G[N(v)]$ is $(\boldk(v), \boldr(v))$-sparse.
\end{defn}

We are now ready to state our result.

\begin{theo}\label{theo:coloring_local_occ}
    There exists $\rho > 0$ such that the following holds for $\mu > 0$ and $\Delta_0$ sufficiently large.
    Let $G$ be a graph of maximum degree $\Delta \geq \Delta_0$ and let $\boldk \,:\, V(G) \to \R$, $\boldr\,:\,V(G) \to \N$, and $\boldeps \,:\, V(G) \to [0, 1]$ be such that $G$ is $(\boldk, \boldr)$-locally-sparse.
    Let $\eps_{\max} = \max_u\boldeps(u)$ and $r_{\max} = \max_u \boldr(u)$.
    Suppose the following hold for each $v \in V(G)$:
    \begin{enumerate}[label= \ep{\normalfont L\arabic*}, leftmargin = \leftmargin + 1\parindent]
        \item\label{item:deg_bound_intro} $\deg(v) \geq \Delta^{\min\set{2\eps_{\max},\, (1+\eps_{\max})/2}}\left(\log(8\Delta^4)\right)^{r_{\max}}$, and
        \item $3 \leq \boldr(v) \leq \rho\log\log \deg(v)/\log\log\log \deg(v)$ and $\boldk(v) = \deg(v)^{\boldeps(v)\boldr(v)}$.
    \end{enumerate}
    If $\mathcal{H} = (L, H)$ is a correspondence cover of $G$ satisfying
    \[\forall u \in V(G), \qquad |L(u)| \geq (1+\mu)\min\left\{2,\,\left(\frac{1 + \eps_{\max}}{1-\eps_{\max}}\right)\right\}\deg(u)\left(\boldeps(u) + \frac{r\log\log\deg(u)}{\log \deg(u)}\right),\]
    then $G$ admits a proper $\mathcal{H}$-coloring.
\end{theo}

A few remarks are in order.
Note that the minimum degree condition \ref{item:deg_bound_intro} is only valid if $\eps_{\max} \leq 1 - o(1)$.
However, for $\eps_{\max} \geq 1/2$, the list sizes implied by the theorem are worse than the greedy bound and so we may assume that $\eps_{\max}$ is not too large in our proofs.
Regardless, the minimum degree condition is rather strict for large values of $\eps_{\max}$ (for instance, $\eps_{\max} = \Omega(1)$).
It turns out we can considerably relax this condition while sacrificing the constant factor involved.

\begin{theo}\label{theo:coloring_local_general_result}
    There exists $\rho > 0$ such that the following holds for $\mu > 0$ and $\Delta$ sufficiently large.
    Let $G$ be a graph of maximum degree $\Delta$ and let $\boldk \,:\, V(G) \to \R$, $\boldr\,:\,V(G) \to \N$, and $\boldeps \,:\, V(G) \to [0, 1]$ be such that $G$ is $(\boldk, \boldr)$-locally-sparse and the following hold for each $v \in V(G)$:
    \begin{enumerate}[label= \ep{\normalfont M\arabic*}, leftmargin = \leftmargin + 1\parindent]
        \item $\deg(v) \geq \log^2\Delta$, and
        \item $3 \leq \boldr(v) \leq \rho\log\log \deg(v)/\log\log\log \deg(v)$ and $\boldk(v) = \deg(v)^{\boldeps(v)\boldr(v)}$.
    \end{enumerate}
    If $\mathcal{H} = (L, H)$ is a correspondence cover of $G$ satisfying
    \[\forall u \in V(G), \qquad |L(u)| \geq (30+\mu)\deg(u)\left(\boldeps(u) + \frac{r\log\log\deg(u)}{\log \deg(u)}\right),\]
    then $G$ admits a proper $\mathcal{H}$-coloring.
\end{theo}

We note that a local correspondence coloring version of Theorem~\ref{theo:dkps} for $r \geq 3$ appeared in work of Davies, Kang, Pirot, and Sereni \cite[Theorem 30 (iii)]{DKPS} (see also \cite[Corollary~1.9]{bonamy2022bounding}).
As a corollary, we obtain a correspondence coloring version of Theorem~\ref{theo:coloring} (see \S\ref{sec:corl} for the proof), which implies Theorem~\ref{theo:coloring} as a result of \eqref{eq:relationship}.

\begin{corl}\label{corl:correspondence}
    There exists $\rho > 0$ such that the following holds for $\Delta$ sufficiently large.
    Let $\eps,\eta \in [0, 1]$, $r \in \N$, and $k \in \R$ be such that
    \[3\,\leq\, r \leq \frac{\rho\,\log\log \Delta}{\log\log\log \Delta}, \quad k \,=\, \Delta^{\eps\,r}, \quad \text{and} \quad \eta = \eps + \frac{r\log\log \Delta}{\log \Delta}.\]
    For any $(k, r)$-locally-sparse graph $G$ of maximum degree $\Delta$, we have
    \[\chi_c(G) \leq (1+o(1))\eta\Delta\min\left\{2,\,\left(\frac{1 + \eps}{1-\eps}\right)\right\}.\]
\end{corl}

\subsection{Concluding Remarks}

In this paper, we generalize classical results of Shearer and Johansson on $K_{r+1}$-free graphs to $(k, r)$-locally-sparse-graphs.
Furthermore, we prove local versions of our coloring result in the setting of correspondence coloring, which extends work of Bonamy, Kelly, Nelson, and Postle, and Davies, Kang, Pirot, and Sereni.
% In the entire paper, we do not attempt to optimize the constant factors involved and leave it as an open problem to do so.
While not explicitly stated in their papers, Theorems~\ref{theo:shearer} and \ref{theo:dkps} hold for $r = O(\log\log d)$ and $r = O(\log\log \Delta)$, respectively.
(For Theorem~\ref{theo:dkps}, this range can be extended to $r = O(\log\Delta)$ with a different bound on $\chi(G)$.)
Our proof fails for larger $r$ and we leave it as an open problem to extend this range.

We conclude this section with a discussion for arbitrary graphs $F$.
A simple counting argument shows the following fact:

\begin{fact}
    Let $G$ be a graph, $k \in \R$ and let $F$ be an arbitrary graph on $r$ vertices.
    If $G$ is $(k, F)$-locally-sparse, then $G$ is $(\tilde k, r)$-locally-sparse, where 
    \[\tilde k = \frac{\lceil k\rceil\,|\mathsf{Aut}(F)|}{r!}.\]
    Here, $\mathsf{Aut}(F)$ is the set of automorphisms of $F$.
\end{fact}

In particular, our results in this section imply the following:

\begin{theo}\label{theo:general_F}
    There exists $\rho > 0$ such that the following holds for $n,\,\Delta$ sufficiently large.
    Let $\eps,\eta \in [0, 1]$, $r \in \N$, and $k \in \R$ be such that
    \[3\,\leq\, r \leq \frac{\rho\,\log\log \Delta}{\log\log\log \Delta}, \quad k \,=\, \Delta^{\eps\,r}, \quad \text{and} \quad \eta = \eps + \frac{r\log\log \Delta}{\log \Delta}.\]
    Let $F$ be an arbitrary graph on $r$ vertices.
    Then for any $n$-vertex $(k, F)$-locally-sparse graph $G$ of maximum degree $\Delta$, we have
    \[\alpha(G) = \Omega\left(\frac{n}{\eta\Delta}\right) \quad \text{and} \quad \chi_c(G) = O\left(\eta\Delta\right),\]
    where the constant factors in the $\Omega(\cdot)$ and $O(\cdot)$ may depend on $F$.
\end{theo}

Note that for any graph $F$ on $r$ vertices, we have $\binom{\Delta}{r}\,r!/|\mathsf{Aut}(F)| \leq \Delta^r$.
In particular, the above result covers all possible values of $k$.
Anderson, Kuchukova, and the author proved a better asymptotic bound on $\chi_c(G)$ when $F$ is bipartite and $k$ is not too large.

\begin{theo}[\cite{local_sparse}]\label{theo:bipartite_F}
    Let $F$ be a bipartite graph and let $\Delta$ be sufficiently large.
    For $k \in \R$ satisfying $1/2 \leq k \leq \Delta^{|V(F)|/10}$, the following holds.
    For any $(k, F)$-locally-sparse graph $G$ of maximum degree $\Delta$, we have
    \[\chi_c(G) \leq \frac{8\,\Delta}{\log\left(\Delta\,k^{-1/|V(F)|}\right)}.\]
\end{theo}

This led the authors to make the following conjecture which Theorem~\ref{theo:general_F} constitutes some progress toward:

\begin{conj}[{\cite[Conjecture~1.20]{local_sparse}}]\label{conj:adk}
    For every graph $F$, the following holds for $\Delta \in \N$ sufficiently large.
    Let 
    \[1/2 \,\leq\, k \,<\, \Delta^{|V(F)|},\]
    and let $G$ be a $(k,F)$-locally-sparse graph of maximum degree $\Delta$.
    Then, 
    \[\chi_c(G) = O\left(\frac{\Delta}{\log\left(\Delta k^{-1/|V(F)|}\right)}\right),\] 
    where the constant factor in the $O(\cdot)$ may depend on $F$.
\end{conj}

The rest of this paper is structured as follows:
in \S\ref{sec:prelim}, we provide an overview of our strategy; in \S\ref{sec:corl}, we prove some of our results that are obtained as corollaries to more general statements;
in \S\ref{sec:ISET_loc_occ}, we prove our main result on the independence number; and in \S\ref{sec:coloring_loc_occ}, we prove our main coloring results.

\section{Preliminaries and overview}\label{sec:prelim}

In this section, we provide an overview of our proof, which implements the techniques developed by Davies, Kang, Pirot, and Sereni in their proof of Theorem~\ref{theo:dkps}.
In particular, we employ the so-called \textit{local occupancy framework}.
We will split this section into two subsections: in \S\ref{sec:local_occ}, we introduce the framework; and in \S\ref{sec:overview}, we discuss the strategy in more detail.

\subsection{The hard-core model and the local occupancy framework}\label{sec:local_occ}

Given a graph $G$, we let $\mathcal{I}(G)$ denote the set of independent sets of $G$.
Given $\lambda > 0$, the \emphd{hard-core model on $G$ at fugacity $\lambda$} is a probability distribution on $\mathcal{I}(G)$, where each independent set $I \in \mathcal{I}(G)$ occurs with probability proportional to $\lambda^{|I|}$ (see \cite{davies2025hard} for a recent survey on the method with a focus on applications to graph coloring).
Formally, writing $\mathbf{I}$ as the random independent set, we have
\[\forall I \in \mathcal{I}(G), \qquad \P[\mathbf{I} = I] = \frac{\lambda^{|I|}}{Z_G(\lambda)},\]
where the normalizing term in the denominator is known as the \emphd{partition function} or \emphd{independence polynomial}.
For $\lambda,\,\beta, \gamma > 0$, we say the hard-core model on $G$ at fugacity $\lambda$ has \emphd{local $(\beta, \gamma)$-occupancy} if for each $u \in V(G)$ and each induced subgraph $F \subseteq G[N(u)]$, we have
\begin{align}\label{eqn: local occupancy}
    \beta\,\frac{\lambda}{1 + \lambda}\,\frac{1}{Z_F(\lambda)} + \gamma\,\frac{\lambda\,Z_F'(\lambda)}{Z_F(\lambda)} \geq 1.
\end{align}

Let us briefly motivate the above.
Let $\mathbf{I}$ be an independent set in $G$ drawn from the hard-core model at fugacity $\lambda$.
For each $u \in V(G)$, define the following random variable:
\[X_u = \left\{ \begin{array}{cc}
    \beta & \text{if } u \in \mathbf{I}; \\
    \gamma|N(u) \cap \mathbf{I}| & \text{if } u \notin \mathbf{I}.
\end{array}\right.\]
(The familiar reader may recognize this strategy as akin to \textit{Shearer-type} arguments for independent sets; see, e.g., \cite{shearer1995independence, alon1996independence}.)
Fix an independent set $J \in \mathcal{I}(G - N[u])$ and let $F_J \coloneqq \set{v \in N(u)\,:\, N(v) \cap J = \emptyset}$, i.e., $F_J$ consists of the vertices $v \in N(u)$ such that $J\cup \set{v}$ is an independent set in $G$.
It is not too difficult to see that the random variable $\mathbf{I} \cap N(u) \mid \mathbf{I} \setminus N[u] = J$ is distributed according to the hard-core model on $G[F_J]$ at fugacity $\lambda$ (see, e.g.,~\cites[p.~16]{DKPS}[\S3]{davies2025hard}).
If the random variable $\mathbf{I} \cap N(u) \mid \mathbf{I} \setminus N[u] = J$ is the empty set, then $u \in \mathbf{I}$ with probability precisely $\lambda/(1+\lambda)$.
In particular, \eqref{eqn: local occupancy} is equivalent to 
\[\forall J \in \mathcal{I}(G - N[u]), \qquad \E[X_u \mid \mathbf{I} \setminus N[u] = J] \geq 1.\]
With this in mind, we can think of \eqref{eqn: local occupancy} as a weighted version of the following statement: \textit{either $u \in \mathbf{I}$ or $|N(u) \cap \mathbf{I}|$ is large}.

The above argument can be formalized in the following result:

\begin{theo}[{cf. \cite{davies2020coloring, davies2021occupancy, davies2017independent, davies2018average}}]\label{theo:loc_occ_ISET}
    Suppose $G$ is a graph of maximum degree $\Delta$ such that the hard-core model on $G$ at fugacity $\lambda$ has local $(\beta, \gamma)$-occupancy for some $\lambda, \beta, \gamma > 0$. 
    Let $\mathbf{I}$ be an independent set in $G$ drawn from the hard-core model at fugacity $\lambda$.
    Then,
    \[\frac{\E[|\mathbf{I}|]}{|V(G)|} = \frac{1}{|V(G)|}\frac{\lambda Z_G'(\lambda)}{Z_G(\lambda)} \geq \frac{1}{\beta + \gamma \Delta}.\]
\end{theo}

The quantity $\frac{1}{|V(G)|}\frac{\lambda Z_G'(\lambda)}{Z_G(\lambda)}$ above is referred to as the \emphd{occupancy fraction} of $G$.
Practically speaking, if one can determine $\beta, \gamma > 0$ such that $G$ has local $(\beta, \gamma)$-occupancy with respect to the hard-core model on $G$, one immediately obtains a bound on the occupancy fraction (and, hence, the independence number).
As mentioned earlier, we obtain Theorem~\ref{theo:independent} as a corollary to a result on the occupancy fraction of locally sparse graphs, which we now state.

\begin{theo}\label{theo:occ_frac}
    There exists $\rho > 0$ such that the following holds for $n,\,\Delta$ sufficiently large.
    Let $\eps,\eta \in [0, 1]$, $r \in \N$, and $k \in \R$ be such that
    \[3\,\leq\, r \leq \frac{\rho\,\log\log \Delta}{\log\log\log \Delta}, \quad k \,=\, \Delta^{\eps\,r}, \quad \text{and} \quad \eta = \eps + \frac{r\log\log \Delta}{\log \Delta}.\]
    For any $n$-vertex $(k, r)$-locally-sparse graph $G$ of maximum degree $\Delta$ and any $\lambda > 0$, the occupancy fraction of the hard-core model on $G$ at fugacity $\lambda$ satisfies the following as $\Delta \to \infty$:
    \[\frac{1}{|V(G)|}\,\frac{\lambda Z_G'(\lambda)}{Z_G(\lambda)} \geq \frac{(1 - o(1))}{\eta\,\Delta}.\]
\end{theo}

Davies, Kang, Pirot, and Sereni showed that if $G$ satisfies certain additional constraints (see Theorem~\ref{theo:DKPS 12}), one can also show that $\chi_c(G) \leq (1 + o(1))\gamma\Delta$.
In fact, one can obtain local coloring results with a more ``local'' notion of local occupancy.

\begin{defn}\label{defn: local occupancy}
    Given a graph $G$, a positive real $\lambda$, and a collection $(\beta_u, \gamma_u)_u$ of pairs of positive reals indexed over the vertices of $G$, we say that the hard-core model on $G$ at fugacity $\lambda$ has \emphd{local $(\beta_u, \gamma_u)_u$-occupancy} if, for each $u \in V (G)$ and each induced subgraph $F$ of the subgraph $G[N(u)]$ induced by the neighborhood of $u$, it holds that
    \[\beta_u\,\frac{\lambda}{1 + \lambda}\,\frac{1}{Z_F(\lambda)} + \gamma_u\,\frac{\lambda\,Z_F'(\lambda)}{Z_F(\lambda)} \geq 1.\]
    We refer to the situation where the above holds for all $F$ (not necessarily induced) as \emphd{strong local $(\beta_u, \gamma_u)_u$-occupancy}.
\end{defn}

\subsection{Proof Overview}\label{sec:overview}

As mentioned earlier, our approach closely follows that of Davies, Kang, Pirot, and Sereni's proof of Theorem~\ref{theo:dkps}.
Their proof proceeds via the following steps:
\begin{enumerate}[label= \ep{\normalfont Step\arabic*}, leftmargin = \leftmargin + 1\parindent]
    \item\label{step:ramsey} Compute a lower bound on $z \defeq \log Z_F(\lambda)$ for $K_r$-free graphs.
    \item\label{step:partition} Use this result to obtain a lower bound on the occupancy fraction for $K_r$-free graphs in terms of $z$.
    \item\label{item:betagamma} Given $\lambda > 0$, determine $\beta,\gamma > 0$ such that the hard-core model at fugacity $\lambda$ on any $K_{r+1}$-free graph $G$ of maximum degree $\Delta$ has local $(\beta, \gamma)$-occupancy.
    \item\label{final_blow} Complete the proof using Theorems~\ref{theo:loc_occ_ISET} and \ref{theo:DKPS 12}.
\end{enumerate}

The key challenges posed in adapting this approach to our setting appear in \ref{step:ramsey}.
This part of their proof relies on a simple upper bound for the off-diagonal Ramsey number $R(s, t)$, however, it can be proved using the following result due to Molloy (which is implicitly used in their proof):

\begin{Lemma}[{\cite[Lemma~13]{molloy2019list}}]\label{lemma:molloy}
    For any $r \geq 3$, let $G$ be a $K_r$-free graph on $n$ vertices.
    Then, 
    \[\alpha(G) \,\geq\, n^{1/(r-1)}.\]
\end{Lemma}

As our goal is to emulate this strategy, we need a lower bound on the independence number of $(k, r)$-sparse graphs $G$. 
As every $(k, r)$-sparse graph $G$ contains a $K_r$-free subgraph $H$ such that $|V(H)| \geq n - k$, we may conclude the following as a corollary to Lemma~\ref{lemma:molloy}:

\begin{Lemma}\label{lemma:bad}
    For any $r \geq 3$, let $G$ be a $(k, r)$-sparse graph on $n$ vertices.
    Then, 
    \[\alpha(G) \,\geq\, (n-k)^{1/(r-1)}.\]
\end{Lemma}

However, this result is not strong enough to prove our results.
In fact, with this lemma, we may prove the following weaker result:

\begin{theo}\label{theo:weak}
    There exists $\rho > 0$ such that the following holds for $n,\,\Delta$ sufficiently large.
    Let $\eps,\eta \in [0, 1]$, $r \in \N$, and $k \in \R$ be such that
    \[3\,\leq\, r \leq \frac{\rho\,\log\log \Delta}{\log\log\log \Delta}, \quad k \,=\, \Delta^{\eps\,r}, \quad \text{and} \quad \eta = \eps\,r + \frac{r\log\log \Delta}{\log \Delta}.\]
    For any $n$-vertex $(k, r)$-locally-sparse graph $G$ of maximum degree $\Delta$, we have
    \[\alpha(G) = \Omega\left(\dfrac{n}{\eta \Delta}\right) \quad \text{and} \quad \chi_c(G) = O\left(\eta\Delta\right).\]
\end{theo}

We do not include a proof of Theorem~\ref{theo:weak}, however, it can be inferred by our arguments.
Note that for $r = O(1)$, the above bounds match those of Theorems~\ref{theo:independent} and \ref{theo:coloring}.
However, for larger $r$, the results of Theorems~\ref{theo:independent} and \ref{theo:coloring} are much stronger.
The key part of our proof, therefore, is a stronger version of Lemma~\ref{lemma:bad} (see Lemma~\ref{lemma:independent set in kr sparse}).

When considering \ref{final_blow} for our coloring result, we apply Theorem~\ref{theo:DKPS 12} just as in the proof of Theorem~\ref{theo:dkps}.
This leads to the drawback mentioned earlier (see the discussion after the statement of Theorem~\ref{theo:coloring_local_occ}).
One of the conditions in Theorem~\ref{theo:DKPS 12} roughly says the following about $G[N(u)]$ for every $u$: \textit{every ``large'' subgraph has ``many'' independent sets}.
There is no freedom in the notion of ``many'', which results in the rather strict minimum degree condition (in particular, we must take $\ell \approx \Delta^{\eps_{\max}}$).
Unfortunately, this notion of ``many'' cannot be relaxed in their current proof as it is a key element in proving certain probability bounds necessary for their application of the \textit{lopsided \LLL}.

A similar result of Bonamy, Kelly, Nelson, and Postle (see Theorem~\ref{theo:BKNP}) has a slightly weaker condition on $G[N(u)]$ for every $u$: \textit{in every subgraph with ``many'' independent sets, half of them are ``large''}.
The key distinction is that we have some control over the parameters that determine the notions of ``many'' and ``large''. 
This allows us to considerably relax the minimum degree condition, however, this comes at the price of a worse constant factor.

\section{Proof of Corollaries}\label{sec:corl}

Let us first show how Theorem~\ref{theo:independent_avg_deg} follows from Theorem~\ref{theo:independent}.

\begin{proof}[Proof of Theorem~\ref{theo:independent_avg_deg}]
    Consider the following sets of vertices:
    \begin{align*}
        V_1 &\defeq \set{v \in V(G)\,:\, \deg(v) \leq 3d}, \\
        V_2 &\defeq \set{v \in V(G)\,:\, G[N(v)] \text{ is } (3d^{\eps\,r}, r)\text{-sparse}}, \\
        V_3 &\defeq V_1\cap V_2.
    \end{align*}
    Let us show that $|V_3| \geq n/3$.
    First, we note the following:
    \[n\,d = \sum_{v \in V(G)} \deg_G(v) \,\geq\, 3d(n - |V_1|) \,\implies\, |V_1| \geq 2n/3.\]
    For each $v \in V(G)$, let $n_v$ denote the number of copies of $K_r$ in $G[N_G(v)]$.
    By definition of $(k, r+1)$-sparsity, we have the following
    \[(r+1)\,k = \sum_{v \in V(G)}n_v \,\geq\, 3d^{\eps\,r}(n - |V_2|) \,\implies\, |V_2| \geq 2n/3.\]
    With the above in hand, we conclude
    \[|V_3| = |V_1| + |V_2| - |V_1\cup V_2| \,\geq\, 2n/3 + 2n/3 - n = n/3,\]
    as claimed.
    
    Let $G'\subseteq G$ be the subgraph induced by the vertices in $V_3$.
    Since an independent set in $G'$ is also independent in $G$, it is enough to show that $G'$ contains an independent set of the desired size.
    Suppose $\Delta(G') \leq 3\eta d$.
    Then we have the following as a result of the greedy bound:
    \[\alpha(G') \geq (1-o(1))\frac{|V_3|}{\Delta(G')} = (1-o(1))\frac{n}{9\,\eta d},\]
    as desired.
    If not, by definition of $V_1$, we have $3\eta d < \Delta(G') \leq 3d$.
    Furthermore, by definition of $V_2$, $G'$ is $(\Delta(G')^{\tilde \eps r}, r)$-locally-sparse, where
    \[\tilde \eps \defeq \min\left\{1,\,\frac{\eps \log d + \log 3}{\log \left(\Delta(G')\right)}\right\}.\]
    As a result of the range of $\Delta(G')$, we may conclude the following for $d$ large enough:
    \begin{align*}
        \eta' &= \tilde \eps + \frac{r\log\log \Delta(G')}{\log \Delta(G')} 
        \leq (1+o(1))\eps + \frac{r\log\log (3\eta d)}{\log(3\eta d)} 
        \leq (1+o(1))\eta.
    \end{align*}
    Applying Theorem~\ref{theo:independent} to $G'$, we have
    \[\alpha(G') \geq (1-o(1))\frac{|V_3|}{\eta'\,\Delta(G')} \geq (1-o(1))\frac{n}{9\,\eta d},\]
    completing the proof.
\end{proof}

The following proposition will be key in proving Corollary~\ref{corl:correspondence}.
Although the proof is rather standard (see, e.g., \cite[\S1.5 and Exercise 12.4]{MolloyReed}), we include the details in \S\ref{appendix:prop} for completeness.

\begin{prop}\label{prop:embedding}
    Let $G$ be a graph of maximum degree $\Delta$ and let $\boldk\,:\,V(G) \to \R$ and $\boldr\,:\,V(G) \to \N$ be such that $G$ is $(\boldk, \boldr)$-locally-sparse.
    For any $\delta \leq \Delta$, there is a graph $G'$ and functions $\tilde \boldk\,:\,V(G') \to \R$ and $\tilde \boldr\,:\,V(G') \to \N$ such that $G'$ is $(\tilde\boldk, \tilde\boldr)$-locally-sparse and the following hold for $j \defeq \max\set{\delta - \delta(G), 0}$:
    \begin{enumerate}[label= \ep{\normalfont I\arabic*}, leftmargin = \leftmargin + 1\parindent]
        \item\label{item:delta} $\delta(G') \geq \delta$,
        \item\label{item:Delta} $\Delta(G') = \Delta$,
        \item\label{item:size} $|V(G')| = |V(G)|\,2^{j}$, and
        \item\label{item:loc_sparsity} for $\ell \defeq 2^{j}$, there exist homomorphisms $\phi_1, \ldots, \phi_\ell\,:\,G\to G'$ such that 
        \begin{itemize}
            \item for each $v \in V(G)$ and $\ell' \in [\ell]$, we have $\boldk(v) = \tilde \boldk (\phi_{\ell'}(v))$ and $\boldr(v) = \tilde \boldr (\phi_{\ell'}(v))$, and
            \item for $\ell_1, \ell_2 \in [\ell]$ such that $\ell_1 \neq \ell_2$, the images $\im(\phi_{\ell_1})$ and $\im(\phi_{\ell_2})$ are vertex-disjoint.
        \end{itemize}
    \end{enumerate}
\end{prop}

We note that we do not require the proposition in its full generality, however, we include it in this form as it may find use in future work on this topic.
Let us now prove Corollary~\ref{corl:correspondence} under the assumption that Theorem~\ref{theo:coloring_local_occ} is true.

\begin{proof}[Proof of Corollary~\ref{corl:correspondence}]
    For each $v \in V(G)$, define 
    \[\boldr(v) \defeq r, \quad \text{and} \quad \boldk(v) \defeq k.\]
    Note that $G$ is $(\boldk, \boldr)$-locally-sparse.
    Form $G'$ by applying Proposition~\ref{prop:embedding} with $\delta = \Delta$.
    Let $\boldeps(v) = \eps$ for each $v \in V(G')$.
    Note that $G'$ satisfies the conditions of Theorem~\ref{theo:coloring_local_occ}.
    In particular, for any correspondence cover $\mathcal{H} = (L, H)$ of $G'$ satisfying 
    \[|L(v)| \geq (1+o(1))\min\left\{2,\,\left(\frac{1 + \eps_{\max}}{1-\eps_{\max}}\right)\right\}\deg_{G'}(v)\left(\boldeps(u) + \frac{r\log\log\deg_{G'}(v)}{\log \deg_{G'}(v)}\right),\]
    $G'$ admits a proper $\mathcal{H}$-coloring.
    As $G'$ is $\Delta$-regular, $\boldeps(v) = \eps$ for all $v \in V(G')$, and $G$ is isomorphic to a subgraph of $G'$ by \ref{item:loc_sparsity}, the claim follows.
\end{proof}

\section{Independent sets in locally sparse graphs}\label{sec:ISET_loc_occ}

In this section, we will prove Theorem~\ref{theo:occ_frac}.
For convenience, we restate it here:

\begin{theo*}[Restatement of Theorem~\ref{theo:occ_frac}]
    There exists $\rho > 0$ such that the following holds for $n,\,\Delta$ sufficiently large.
    Let $\eps,\eta \in [0, 1]$, $r \in \N$, and $k \in \R$ be such that
    \[3\,\leq\, r \leq \frac{\rho\,\log\log \Delta}{\log\log\log \Delta}, \quad k \,=\, \Delta^{\eps\,r}, \quad \text{and} \quad \eta = \eps + \frac{r\log\log \Delta}{\log \Delta}.\]
    For any $n$-vertex $(k, r)$-locally-sparse graph $G$ of maximum degree $\Delta$ and any $\lambda > 0$, the occupancy fraction of the hard-core model on $G$ at fugacity $\lambda$ satisfies the following as $\Delta \to \infty$:
    \[\frac{1}{|V(G)|}\,\frac{\lambda Z_G'(\lambda)}{Z_G(\lambda)} \geq \frac{(1 - o(1))}{\eta\,\Delta}.\]
\end{theo*}

En route to the above result, we will prove some auxiliary lemmas that will be crucial to our arguments in \S\ref{sec:coloring_loc_occ}.
We will split this section into two subsections:
in the first, we will analyze the hard-core model for $(k, r)$-sparse graphs;
in the next, we will determine the local occupancy parameters for locally sparse graphs and complete the proof of Theorem~\ref{theo:occ_frac}.

\subsection{The hard-core model for $(k, r)$-sparse graphs}

The random greedy bound on $\alpha(G)$ is a classical result due to Tur\'an, which will play a key role in proving our bound on the independence number of $(k, r)$-sparse graphs.
For a proof of the following statement, see the section following chapter 6 of \cite{alon2016probabilistic}:

\begin{Lemma}\label{lemma:greedy}
    For any $n$-vertex graph $G$ of average degree $d$, we have $\alpha(G) \geq \dfrac{n}{1+d}$.
\end{Lemma}

The next lemma provides a general lower bound on the independence number of $(k, r)$-sparse graphs and will be key in proving the main results of this paper as mentioned in \S\ref{sec:overview}.
We note that our main results only consider the case that $r \geq 3$, however, including the case $r = 2$ below provides for a simpler proof.

\begin{Lemma}\label{lemma:independent set in kr sparse}
    Let $r, n \in \N$ and $k \in \R$ such that $r\geq 2$, $n \geq r^{2r}$, and $k \geq 1$.
    Let $G$ be an $n$-vertex $(k, r)$-sparse graph. Then,
    \[\alpha(G) \,\geq\, \frac{1}{r}\left(\dfrac{n}{k^{1/r}}\right)^{1/(r-1)}.\]
\end{Lemma}

\begin{proof}
    We may assume without loss of generality that $k \leq \binom{n}{r}$ as the proof is trivial otherwise.
    We will prove this by induction on $r$.
    
    As a base case, let $r = 2$.
    We remark that it is enough to have $n \geq 3$ here, however, by the assumptions of the lemma, we have $n \geq 16$.
    By Lemma~\ref{lemma:greedy}, we have an independent set of size at least 
    \[\frac{n}{1 + d(G)}.\]
    As $G$ is $(k, 2)$-sparse, we have $|E(G)| \leq k$.
    It follows that 
    \[\frac{n}{1 + d(G)} \,\geq\, \frac{n}{1 + 2k/n}.\]
    It is now enough to show that the above is at least $n/(2\sqrt{k})$.
    To this end, we consider the following function:
    \[f(k) = 2\sqrt{k} - 2k/n - 1.\]
    Note that $f(k) \geq 0$ for $1 \leq k \leq \binom{n}{2}$ implies the desired result.
    Consider the following for $n \geq 16$:
    \begin{align*}
        f(1) &= 2 - 2/n - 1 > 0, \\
        f(n^2/2) &= (\sqrt{2} - 1)n - 1 > 0, \\
        f'(k) &= k^{-1/2} - 2/n \geq 0 \iff k \leq \frac{n^2}{4}.
    \end{align*}
    It follows that $f$ is increasing for $1 \leq k \leq n^2/4$ and decreasing for $n^2/4 < k \leq n^2/2$.
    As $f(1), f(n^2/4), f(n^2/2) > 0$, we conclude $f(k) > 0$ for $1 \leq k \leq \binom{n}{2}$, as desired.

    Now let us consider $r \geq 3$.
    We define the following parameters:
    \begin{align*}
        \alpha_r &\defeq \left(\frac{r-1}{r}\right)^{r - 2 - \frac{1}{r-1}}, & \beta_r &\defeq r - 1, \\
        X &\defeq \alpha_r\,n^{\frac{r-2}{r-1}}\,k^{\frac{1}{r(r-1)}}, & B &\defeq \set{v\in V(G)\,:\, \deg_G(v) \geq X}.
    \end{align*}
    The following bound on $X$ follows as $r \geq 3$, $n \geq r^{2r}$, and $k \geq 1$:
    \begin{align*}
        X &= \left(\frac{r-1}{r}\right)^{r - 2 - \frac{1}{r-1}}\,n^{\frac{r-2}{r-1}}\,k^{\frac{1}{r(r-1)}} \\
        &\geq \left(\frac{r-1}{r}\right)^{r - 2 - \frac{1}{r-1}}\,r^{\frac{2r(r-2)}{r-1}} \\
        &= (r-1)^{2(r-1)- r - \frac{1}{r-1}}\,r^{\frac{2r(r-2) + 1}{r-1} - (r - 2)} \\
        &= (r-1)^{2(r-1)}\,(r-1)^{- r - \frac{1}{r-1}}\,r^{\frac{(r+1)(r-2) + 1}{r-1}} \\
        &= (r-1)^{2(r-1)}\,(r-1)^{- r - \frac{1}{r-1}}\,r^{\frac{r^2 - r - 1}{r-1}} \\
        &= (r-1)^{2(r-1)}\,\left(\frac{r}{r-1}\right)^{r}\,\frac{1}{(r(r-1))^{1/(r-1)}}.
    \end{align*}
    For $r = 3$, it can be verified that
    \[\left(\frac{r}{r-1}\right)^{r}\,\frac{1}{(r(r-1))^{1/(r-1)}} \geq 1.\]
    For $r \geq 4$, we have
    \begin{align*}
        \log(r(r-1)) \,\leq\, 2\log r \,\leq\, r - 1.
    \end{align*}
    It follows that
    \[(r(r-1))^{1/(r-1)} \leq e.\]
    Furthermore, we have
    \[(1 - 1/r)^r \leq \frac{1}{e} \implies \frac{1}{(1 - 1/r)^r} \geq e,\]
    In particular, we may conclude that
    \begin{align}\label{eq:X_bound}
        X \geq (r-1)^{2(r-1)}.
    \end{align}
    We will consider two cases based on the size of $B$.
    \begin{enumerate}[label=\ep{\textbf{Case\arabic*}}, leftmargin = \leftmargin + 1\parindent]
        \item\label{item:case1} $|B| \geq \beta_r\,k^{1/r}$.
        For each $v \in B$, let $n_v$ denote the number of copies of $K_r$ in $G$ containing $v$.
        By definition of $(k, r)$-sparsity, we have:
        \[\sum_{v \in B}n_v \leq rk,\]
        implying there is some $v \in B$ such that $n_v \leq rk/|B| \leq rk^{1-1/r}/\beta_r$.
        In particular, the graph $H \defeq G[N(v)]$ is $(\tilde k, r-1)$-sparse, where
        \[\tilde k \defeq \frac{r\,k^{1 - 1/r}}{\beta_r}.\]
        Note the following:
        \begin{align*}
            \frac{r\,k^{1 - 1/r}}{\beta_r} &\geq\, \frac{r}{(r-1)} \,>\, 1.
        \end{align*}
        Therefore, we conclude $\tilde k \geq 1$ and $|V(H)| \geq (r-1)^{2(r-1)}$ (by \eqref{eq:X_bound}).
        By the induction hypothesis, there is an independent set in $H$ of size at least
        \begin{align*}
            \frac{1}{r-1}\left(\frac{|V(H)|}{\tilde k^{1/(r-1)}}\right)^{1/(r-2)} &\geq \frac{1}{r-1}\left(\frac{X}{\left(r\,k^{1-1/r}/\beta_r\right)^{1/(r-1)}}\right)^{1/(r-2)} \\
            &= \frac{1}{r-1}\left(\frac{n}{k^{1/r}}\right)^{1/(r-1)}\left(\frac{\alpha_r^{r-1}\beta_r}{r}\right)^{1/((r-1)(r-2))} \\
            &= \frac{1}{r}\left(\frac{n}{k^{1/r}}\right)^{1/(r-1)},
        \end{align*}
        as desired.
        (We note that as $r \geq 3$, $1/(r-2)$ is well defined.)
        As this set is also independent in $G$, this concludes \ref{item:case1}.

        \item $|B| < \beta_r\,k^{1/r}$.
        First, we note the following:
        \[\frac{|B|}{X} \,<\, \frac{\beta_r\,k^{1/r}}{X} \,=\, \frac{\beta_r}{\alpha_r}\left(\frac{k^{1/r}}{n}\right)^{\frac{r-2}{r-1}} \,\leq\, \frac{\beta_r}{\alpha_r},\]
        where we use the fact that $k \leq \binom{n}{r} \leq n^r$.
        As a result, we have
        \begin{align*}
            d(G) \,\leq\, \frac{1}{n}\left(|B|n + (n - |B|)X\right) &\,\leq\, |B| + X \\
            &\leq \frac{X}{\alpha_r}\left(\beta_r + \alpha_r\right) \\
            &= \frac{rX}{\alpha_r} - \frac{X}{\alpha_r}\left(1 - \alpha_r\right).
        \end{align*}
        Let us consider the term on the right.
        We have
        \begin{align*}
            \frac{X}{\alpha_r}\left(1 - \alpha_r\right) &= n^{\frac{r-2}{r-1}}\,k^{\frac{1}{r(r-1)}}\left(1 - \left(1 - \frac{1}{r}\right)^{r - 2 - \frac{1}{r-1}}\right) \\
            &\geq n^{1/2}\left(1 - \exp\left(-\frac{1}{r}\left(r - 2 - \frac{1}{r-1}\right)\right)\right) \\
            &= n^{1/2}\left(1 - \exp\left(-1 + \frac{2}{r} + \frac{1}{r(r-1)}\right)\right) \\
            &\geq n^{1/2}\left(1 - \exp(-1/6)\right) > 1,
        \end{align*}
        where we use the fact that $n \geq r^{2r}$ and $r \geq 3$.
        In particular, we have
        \[d(G) \,\leq\, \frac{rX}{\alpha_r} - 1.\]
        By Lemma~\ref{lemma:greedy}, it follows that $G$ has an independent set of size at least
        \[\frac{n}{1 + d(G)} \quad\geq\quad \frac{\alpha_r\,n}{rX} \quad=\quad \frac{1}{r}\left(\frac{n}{k^{1/r}}\right)^{1/(r-1)},\]
        as desired.
    \end{enumerate}
    This covers all cases, completing the proof.
\end{proof}

With this lemma in hand, we are ready to complete \ref{step:ramsey} of our proof.

\begin{Lemma}\label{lemma: z_lb}
    Let $r, n \in \N$ and $k \in \R$ such that $r\geq 3$, $n \geq r^{2r}$, and $k \geq 1$.
    Let $G$ be an $n$-vertex $(k, r)$-sparse graph.
    For any positive integer $\alpha$ and $\lambda > 0$, we have
    \[\log Z_G(\lambda) \geq \alpha\left(\log(n\lambda) - \frac{\log k}{r} - (r-1)\log (r\alpha)\right).\]
\end{Lemma}

\begin{proof}
    Let $t \defeq t(k, r, \alpha)$ be the minimum integer such that every $(k,r)$-sparse graph on $t$ vertices has an independent set of size at least $\alpha$.
    By Lemma~\ref{lemma:independent set in kr sparse}, we know that $t \leq k^{1/r}(r\alpha)^{r-1}$.

    As $G$ is $(k, r)$-sparse, we may conclude that $G$ contains at least
    \[\dfrac{\binom{n}{t}}{\binom{n-\alpha}{t-\alpha}} \geq \left(\frac{n}{t}\right)^\alpha\]
    independent sets of size $\alpha$.

    It follows that
    \[Z_G(\lambda) \geq \left(\frac{n\lambda}{t}\right)^\alpha.\]
    The claim now follows by the upper bound on $t$ described earlier.
\end{proof}

\begin{Lemma}\label{lemma:partition_fn_bound}
    Let $r, n \in \N$ and $k \in \R$ such that $r\geq 3$, $n \geq r^{2r}$, and $k \geq 1$.
    Let $G$ be an $n$-vertex $(k, r)$-sparse graph.
    For any $\lambda > 0$ and $z = \log Z_G(\lambda)$, we have
    \[\frac{\lambda\,Z_G'(\lambda)}{Z_G(\lambda)} \geq \frac{1 - o_n(1)}{r - 2}\,\frac{z}{\log\left(e^{\frac{r}{r-2}}\,k^{\frac{1}{r(r-2)}}\,z\right)}.\]
\end{Lemma}

\begin{proof}
    First, let us apply Lemma~\ref{lemma: z_lb} with
    \[\alpha = \frac{1}{re^{r/(r-1)}}\left(\frac{n\lambda}{k^{1/r}}\right)^{1/(r-1)}.\]
    From here, it follows that
    \begin{align}
        z &\geq r\alpha = \frac{1}{e^{r/(r-1)}}\left(\frac{n\lambda}{k^{1/r}}\right)^{1/(r-1)} \label{eq: z_bound} \\
        &\implies \frac{n\lambda}{z} \leq e^{r}\,k^{1/r}\,z^{r-2}. \nonumber
    \end{align}
    For fixed $\lambda$, the RHS above tends to $\infty$ as $n$ (and hence, $z$) tends to $\infty$.
    The claim now follows by \cite[Lemma 19]{DKPS} (we note that the function $K(\cdot)$ in \cite[Lemma 19]{DKPS} satisfies $K(x) = (1+o(1))\log x$).
\end{proof}

\subsection{Local occupancy for $(\boldk, \boldr)$-locally-sparse graphs}

The main result of this section determines the (strong) local occupancy parameters for $(\boldk, \boldr)$-locally-sparse graphs, which will be crucial to our arguments in \S\ref{subsec:dkps} as well.

\begin{Lemma}\label{lemma:local_occ_parameters}
    For any $\lambda, \xi > 0$, there exists $\rho > 0$ such that the following holds for $n,\,d_0$ sufficiently large.
    Let $G$ be a graph and let $\boldk \,:\, V(G) \to \R$, $\boldr\,:\,V(G) \to \N$, and $\boldeps \,:\, V(G) \to [0, \infty)$ be such that $G$ is $(\boldk, \boldr)$-locally-sparse.
    Suppose additionally that the following hold for some positive reals $(d_u)_{u \in V(G)}$ and every vertex $v \in V(G)$:
    \begin{enumerate}[label= \ep{\normalfont S\arabic*}, leftmargin = \leftmargin + 1\parindent]
        \item\label{item:deg_bound_loc_occ} $d_v \geq d_0$, and
        \item\label{item:k_r_bound_loc_occ} $3 \leq \boldr(v) \leq \rho\log\log d_v/\log\log\log d_v$ and $\boldk(v) = d_v^{\boldeps(v)\boldr(v)}$.
    \end{enumerate}
    Then there is a choice of parameters $(\beta_u, \gamma_u)_{u \in V(G)}$ that satisfy
    \[\forall u \in V(G), \qquad \beta_u + \gamma_ud_u \leq (1+\xi)^2\,d_u\left(\boldeps(u) + \frac{\boldr(u)\log\log d_u}{\log d_u}\right)\]
    and strong local $(\beta_u, \gamma_u)_u$-occupancy in the hard-core model on $G$ at fugacity $\lambda$.
\end{Lemma}

\begin{proof}
    Let $\sigma \in (0, 1)$ to be specified later.
    Fix a vertex $u$ and let $F$ be a subgraph of $G[N(u)]$ on $t$ vertices.
    We will write $z = \log Z_F(\lambda)$, $r = \boldr(u)$, $k = \boldk(u)$, and $\eps = \boldeps(u)$ for brevity.

    Recall the definition of strong local occupancy from Definition~\ref{defn: local occupancy}. In particular, we must find choices of $\beta_u$ and $\gamma_u$ such that 
    \[\beta_u\frac{\lambda}{1 + \lambda}e^{-z} + \gamma_u \frac{\lambda Z_F'(\lambda)}{Z_F(\lambda)} \geq 1.\]
    The following function will assist with our proof:
    \[g(z) = \beta_u\frac{\lambda}{1 + \lambda}e^{-z} + \gamma_u \frac{1 - \sigma}{r - 2}\,\frac{z}{\log\left(\omega\,z\right)}, \qquad \text{where} \qquad \omega = e^{\frac{r}{r-2}}\,k^{\frac{1}{r(r-2)}}.\]
     % $\tilde k = \min\set{k, \binom{t}{r}}$.
    Furthermore, let $t_0 \geq r^{2r}$ to be specified later.
    We claim that it is enough to choose $\beta_u$ and $\gamma_u$ such that
    \begin{align}
        g(z) &\geq 1 \qquad \text{whenever } t > t_0, \label{eq:g_main} \\
        \beta_u &\geq \frac{(1 + \lambda)^{1+ t_0}}{\lambda}. \label{eq:small_t}
    \end{align}
    Indeed, when $t \leq t_0$, $z \leq t_0 \log (1+\lambda)$ trivially and so the claim follows by \eqref{eq:small_t}.
    When $t > t_0$, the claim follows by \eqref{eq:g_main} and Lemma~\ref{lemma:partition_fn_bound}.

    For $t_0 < t \leq \deg(u)$, we have $z \in I = [\log(1 + t_0\lambda),\, \deg(u)\log (1 + \lambda)]$.
    Therefore, in order to establish \eqref{eq:g_main}, it is enough to investigate the minimum value of $g$ on this interval.
    Let $z^* \geq e/\omega$ be given by the following equation:
    \[d_u \frac{\lambda}{1 + \lambda}e^{-z^*} = \frac{1 - \sigma}{r - 2}\,\frac{z^*}{\log\left(\omega\,z^*\right)}.\]
    We claim that $z^*$ is unique.
    First, note that the LHS is decreasing in $z^*$ and the RHS is increasing for $z^* \geq e/\omega$.
    Furthermore, the LHS is strictly larger than the RHS for $d_0$ sufficiently large in terms of $\lambda$ and $\sigma$.
    Letting
    \[\tau_u = d_u\frac{\lambda}{1 + \lambda}\,\frac{r - 2}{1 - \sigma},\]
    we have
    \[z^* = \log \tau_u - \log\log\tau_u + \log\log (\omega\log \tau_u) + o(1),\]
    as $d_u$ (equivalently $\tau_u$) tends to infinity.

    We define $\beta_u$ and $\gamma_u$ by first solving for the parameters in the equations $g'(z^*) = 0$ and $g(z^*) = 1$, and then multiplying the result by $1/(1 - \sigma)$.
    Standard calculus arguments yield
    \[\beta_u = d_u\frac{r - 2}{(1 - \sigma)^2}\,\frac{\log(\omega z^*)(\log (\omega z^*) - 1)}{z^*((z^* + 1)\log(\omega z^*) - 1)}, \qquad \text{and} \qquad \gamma_u = \frac{r - 2}{(1 - \sigma)^2}\,\frac{\log^2(\omega z^*)}{(z^* + 1)\log(\omega z^*) - 1}.\]
    As $d_u \to \infty$, we have
    \begin{align*}
        \beta_u &\sim d_u\frac{r - 2}{(1 - \sigma)^2}\,\frac{\log(\omega z^*)}{(z^*)^2} \sim d_u\frac{r - 2}{(1 - \sigma)^2}\,\frac{\log(\omega \log ((r - 2)d_u))}{(\log ((r-2)d_u))^2}, \\
        \gamma_u &\sim \frac{r - 2}{(1 - \sigma)^2}\,\frac{\log(\omega z^*)}{z^*} \sim \frac{r - 2}{(1 - \sigma)^2}\,\frac{\log(\omega \log ((r - 2)d_u))}{\log ((r - 2)d_u)},
    \end{align*}
    implying
    \[\beta_u + \gamma_ud_u \sim d_u\frac{r - 2}{(1 - \sigma)^2}\,\frac{\log(\omega \log ((r - 2)d_u))}{\log ((r - 2)d_u)}.\]
    Setting $t_0 = \frac{\log d_u}{2\log(1+\lambda)}$, we satisfy \eqref{eq:small_t} (note that $t_0 \geq r^{2r}$ by \ref{item:k_r_bound_loc_occ}).
    Furthermore, choosing $\sigma$ sufficiently small in terms of $\xi$, we have
    \[\beta_u + \gamma_ud_u \leq (1+\xi)d_u\,(r - 2)\,\frac{\log(\omega \log d_u)}{\log d_u}.\]

    It remains to justify that $z^*$ is the minimizer of $g$ on the interval $I$.
    Note that we may assume that $z^* \in I$ by extending the interval if needed and so it is sufficient to check the endpoints of $I$.
    For $d_0$ sufficiently large, we have
    \[g(\log(1 + t_0\lambda)) > \beta_u\frac{\lambda}{(1+\lambda)(1+t_0\lambda)} > 1.\]
    As $z^* \in I$, we may assume that $\deg(u)\log (1+\lambda) \geq z^*$, implying
    \begin{align*}
        g(\deg(u)\log (1+\lambda)) &> \gamma\,\frac{1 - \sigma}{r - 2}\,\frac{\deg(u)\log (1+\lambda)}{\log\left(\omega\,\deg(u)\log (1+\lambda)\right)} \\
        &\geq \gamma\,\frac{1 - \sigma}{r - 2}\,\frac{z^*}{\log\left(\omega\,z^*\right)} \\
        &\sim \frac{1}{1 - \sigma} > 1,
    \end{align*}
    for $d_0$ sufficiently large.

   With this in hand, we have
    \begin{align*}
        \beta_u + \gamma_ud_u &\leq (1+\xi)\,d_u\left(\eps + \frac{(r-2)\log \log d_u + r}{\log d_u}\right) \\
        &\leq (1+\xi)^2\,d_u\left(\eps + \frac{r\log \log d_u}{\log d_u}\right),
    \end{align*}
    for $d_0$ sufficiently large in terms of $\lambda$ and $\xi$.
\end{proof}

In order to prove Theorem~\ref{theo:occ_frac}, note that for a fixed $\lambda > 0$ we may apply Lemma~\ref{lemma:local_occ_parameters} with $d_u = \Delta$ for each $u \in V(G)$, and an arbitrary $\xi > 0$.
The claim then follows by Theorem~\ref{theo:loc_occ_ISET}.

\section{Coloring locally sparse graphs}\label{sec:coloring_loc_occ}

In this section, we will prove our coloring results.
We will further split this section into two subsections devoted to the proofs of Theorems~\ref{theo:coloring_local_occ} and \ref{theo:coloring_local_general_result}, respectively.

\subsection{Proof of Theorem~\ref{theo:coloring_local_occ}}\label{subsec:dkps}

For convenience, we restate the result.

\begin{theo*}[Restatement of Theorem~\ref{theo:coloring_local_occ}]
    There exists $\rho > 0$ such that the following holds for $\mu > 0$ and $\Delta_0$ sufficiently large.
    Let $G$ be a graph of maximum degree $\Delta \geq \Delta_0$ and let $\boldk \,:\, V(G) \to \R$, $\boldr\,:\,V(G) \to \N$, and $\boldeps \,:\, V(G) \to [0, 1]$ be such that $G$ is $(\boldk, \boldr)$-locally-sparse.
    Let $\eps_{\max} = \max_u\boldeps(u)$ and $r_{\max} = \max_u \boldr(u)$.
    Suppose the following hold for each $v \in V(G)$:
    \begin{enumerate}[label= \ep{\normalfont L\arabic*}, leftmargin = \leftmargin + 1\parindent]
        \item\label{item:deg_bound} $\deg(v) \geq \Delta^{\min\set{2\eps_{\max},\, (1+\eps_{\max})/2}}\left(\log(8\Delta^4)\right)^{r_{\max}}$, and
        \item\label{item:k_r_bound} $3 \leq \boldr(v) \leq \rho\log\log \deg(v)/\log\log\log \deg(v)$ and $\boldk(v) = \deg(v)^{\boldeps(v)\boldr(v)}$.
    \end{enumerate}
    If $\mathcal{H} = (L, H)$ is a correspondence cover of $G$ satisfying
    \[\forall u \in V(G), \qquad |L(u)| \geq (1+\mu)\min\left\{2,\,\left(\frac{1 + \eps_{\max}}{1-\eps_{\max}}\right)\right\}\deg(u)\left(\boldeps(u) + \frac{r\log\log\deg(u)}{\log \deg(u)}\right),\]
    then $G$ admits a proper $\mathcal{H}$-coloring.
\end{theo*}

Throughout the proof we will assume $\mu$ is sufficiently small.
As mentioned in \S\ref{sec:overview}, the bound on $|L(u)|$ is worse than the greedy bound for $\eps_{\max} > 1/2$ and so we will assume $\eps_{\max} \leq 1/2$ for the proof.
To prove the above, we will apply the following result of Davies, Kang, Pirot, and Sereni:

\begin{theo}[{\cite[Theorem 12]{DKPS}}]\label{theo:DKPS 12}
    Suppose that $G$ is a graph of maximum degree $\Delta \geq 2^6$ such that the hard-core model on $G$ at fugacity $\lambda$ has local $(\beta_u, \gamma_u)_u$-occupancy for some $\lambda > 0$ and some collection $(\beta_u, \gamma_u)_u$ of positive reals.
    Suppose also for some $\ell > \log \Delta$ that we are given a correspondence cover $\mathcal{H} = (L, H)$ of $G$ that arises from a list-assignment of $G$, that satisfies
    \[|L(u)| \geq \beta_u\frac{\lambda}{1 + \lambda}\frac{\ell}{1 - \sqrt{(7\log \Delta)/\ell}} + \gamma\deg(u),\]
    for all $u \in V(G)$ and
    \[Z_F(\lambda) \geq 8\Delta^4,\]
    for all induced subgraphs $F \subseteq G[N(u)]$ on at least $\ell/8$ vertices.
    Then $G$ admits an $\mathcal{H}$-coloring.

    If $\mathcal{H}$ does not arise from a list-assignment, then strong local $(\beta_u, \gamma_u)_u$-occupancy is sufficient for an $\mathcal{H}$-coloring.
\end{theo}

Fix an arbitrary $\lambda > 0$ and let $\xi$ be such that $(1+\xi)^3 < 1 + \mu$.
Consider an arbitrary $F \subseteq G[N(u)]$ on $t \geq r^{2r}$ vertices, where $r = \boldr(u)$.
Clearly, $F$ is $(k, r)$-sparse, where $k = \boldk(u)$.
From \eqref{eq: z_bound}, we have
\[\log Z_F(\lambda) \geq \frac{1}{e^{r/(r-1)}}\left(\frac{t\lambda}{k^{1/r}}\right)^{1/(r-1)}.\]
It follows that the condition on $Z_F(\lambda)$ is satisfied for $\ell$ defined as follows:
\[\ell \defeq \frac{k_{\max}^{1/r_{\max}}e^{r_{\max}}\left(\log (8\Delta^4)\right)^{r_{\max} - 1}}{\lambda},\]
where $k_{\max} = \max_u\boldk(u)$.
(Note that as $r_{\max} \geq 3$ and $k_{\max}\geq 1$, we have $\ell = \omega(\log \Delta)$ as $\Delta \to \infty$.)

The goal is to apply Lemma~\ref{lemma:local_occ_parameters} with
\[d_u = \frac{\deg(u)}{\frac{\lambda}{1 - \lambda}\frac{\ell}{1 - \sqrt{(7\log \Delta)/\ell}}}\]
To this end, we note the following:
\begin{align*}
    d_u &= (1-o(1))\frac{\deg(u)}{\frac{\lambda\,\ell}{1 - \lambda}} \\
    &\geq (1-o(1))(1 - \lambda)\frac{\deg(u)}{k_{\max}^{1/r_{\max}}e^{r_{\max}}\left(\log (8\Delta^4)\right)^{r_{\max} - 1}} \\
    &\geq (1-o(1))(1 - \lambda)\frac{\deg(u)}{\Delta^{\eps_{\max}}e^{r_{\max}}\left(\log (8\Delta^4)\right)^{r_{\max} - 1}} \\
    &\geq \Delta^{(1 - o(1))\min\set{\eps_{\max},(1-\eps_{\max})/2}}.
\end{align*}
In particular, $d_u$ is at least $d_0$ for $\Delta_0$ sufficiently large.
Similarly, we may conclude that
\[d_u \geq \deg(u)^{(1-o(1))\max\{1/2,\,(1-\eps_{\max})/(1+\eps_{\max})\}}.\]
In particular, we may apply Lemma~\ref{lemma:local_occ_parameters} with $(d_u)_u$ and $\boldr$ as defined, along with the following functions:
\begin{align*}
    \tilde{\boldeps}(u) &= (1+o(1))\boldeps(u)\min\left\{2,\,\left(\frac{1 + \eps_{\max}}{1-\eps_{\max}}\right)\right\}, \\
    \tilde{\boldk}(u) &= d_u^{\tilde{\boldeps(u)}}.
\end{align*}
To complete the proof, it is now enough to show that
\[|L(u)| \geq \frac{\lambda}{1+\lambda}\frac{\ell}{1 - \sqrt{(7\log \Delta)/\ell}}(\beta_u + \gamma_ud_u),\]
for all $u \in V(G)$.
To this end, we note the following:
\begin{align*}
    &\qquad\frac{\lambda}{1+\lambda}\frac{\ell}{1 - \sqrt{(7\log \Delta)/\ell}}(\beta_u + \gamma_ud_u) \\
    &\leq \frac{\lambda}{1+\lambda}\frac{\ell}{1 - \sqrt{(7\log \Delta)/\ell}}(1+\xi)^2d_u\left(\tilde{\boldeps}(u) + \frac{\log \log d_u}{\log d_u}\right) \\
    &\leq (1+\xi)^3\min\left\{2,\,\left(\frac{1 + \eps_{\max}}{1-\eps_{\max}}\right)\right\}\deg(u)\left(\boldeps(u) + \frac{\log \log \deg(u)}{\log \deg(u)}\right),
\end{align*}
completing the proof.

\subsection{Proof of Theorem~\ref{theo:coloring_local_general_result}}\label{subsec:bknp}

For convenience, we restate the result.

\begin{theo*}[Restatement of Theorem~\ref{theo:coloring_local_general_result}]
    There exists $\rho > 0$ such that the following holds for $\mu > 0$ and $\Delta$ sufficiently large.
    Let $G$ be a graph of maximum degree $\Delta$ and let $\boldk \,:\, V(G) \to \R$, $\boldr\,:\,V(G) \to \N$, and $\boldeps \,:\, V(G) \to [0, 1]$ be such that $G$ is $(\boldk, \boldr)$-locally-sparse and the following hold for each $v \in V(G)$:
    \begin{enumerate}[label= \ep{\normalfont M\arabic*}, leftmargin = \leftmargin + 1\parindent]
        \item\label{item:deg_bound_median} $\deg(v) \geq \log^2\Delta$, and
        \item\label{item:k_r_bound_median} $3 \leq \boldr(v) \leq \rho\log\log \deg(v)/\log\log\log \deg(v)$ and $\boldk(v) = \deg(v)^{\boldeps(v)\boldr(v)}$.
    \end{enumerate}
    If $\mathcal{H} = (L, H)$ is a correspondence cover of $G$ satisfying
    \[\forall u \in V(G), \qquad |L(u)| \geq (30+\mu)\deg(u)\left(\boldeps(u) + \frac{r\log\log\deg(u)}{\log \deg(u)}\right),\]
    then $G$ admits a proper $\mathcal{H}$-coloring.
\end{theo*}

Throughout the proof we will assume $\mu$ is sufficiently small.
As mentioned in \S\ref{sec:overview}, we will apply a result of Bonamy, Kelly, Nelson, and Postle.
Before we do so, we need an auxiliary result on the \emphd{median independence number} (denoted $\overline{\alpha}(G)$) of $(k,r)$-sparse graphs defined as follows:
\begin{align*}
    \overline{\alpha}(G) &\defeq \max\left\{\ell \in \N\,:\,|\set{I \in \mathcal{I}(G)\,:\, |I| \geq \ell}| \geq \frac{i(G)}{2}\right\},
\end{align*}
where $i(G) = |\mathcal{I}(G)|$.

\begin{Lemma}\label{lemma:median}
    Let $r,\, n \in \N$ and $k \in \R$ such that $r\geq 2$, $n \geq r^{2r}$, and $k \geq 1$.
    Let $G$ be an $n$-vertex $(k, r)$-sparse graph.
    Then,
    \[\overline{\alpha}(G) \geq \frac{\log\left(i(G)\right)}{2r\log\left(r\,k^{1/(r(r-1))}\log\left(i(G)\right)\right)}.\]
\end{Lemma}

\begin{proof}
    It is enough to show that at most half of the independent sets in $i(G)$ have size at most 
    \[\ell \defeq \frac{\log\left(i(G)\right)}{2(r-1)\log\left(r\,k^{1/(r(r-1))}\log\left(i(G)\right)\right)}.\]
    In fact, we will show something stronger, i.e.,
    \[\sum_{i = 1}^\ell \binom{n}{i} \leq \frac{i(G)}{2}.\]
    If $\ell < 1$, the claim is trivial, so we may assume $\ell \geq 1$.
    Note the following:
    \[\log\left(i(G)\right) \,\leq\, n \,\leq\, k^{1/r}\left(r\log\left(i(G)\right)\right)^{r-1}.\]
    (The first inequality is trivial and the second follows due to Lemma~\ref{lemma:independent set in kr sparse}.)
    In particular, the first inequality implies that $\ell \leq n/2$ (assuming $i(G) \geq 2$, which holds for $n \geq 2$).
    With this in hand, we have
    \begin{align*}
        \sum_{i = 1}^\ell \binom{n}{i} &\leq \sum_{i = 1}^\ell n^i 
        \leq \ell\,n^\ell 
        \leq \frac{1}{2}\,n^{2\ell} 
        \\ &\leq \frac{1}{2}\left(k^{1/r}\left(r\log\left(i(G)\right)\right)^{r-1}\right)^{2\ell}
        \\ &\leq \frac{1}{2}\,x^{2(r-1)\ell},
    \end{align*}
    for $x \defeq r\,k^{1/(r(r-1))}\log\left(i(G)\right)$.
    From here, it follows that
    \[x^{2(r-1)\ell} \quad=\quad e^{2(r-1)\ell\log x} \quad=\quad i(G),\]
    completing the proof.
\end{proof}

We are now ready to apply the following result of Bonamy, Kelly, Nelson, and Postle:

\begin{theo}[{\cite[Theorem~1.13]{bonamy2022bounding}}]\label{theo:BKNP}
    There exists $\Delta_0 \in \N$ such that the following holds.
    Let $G$ be a graph of maximum degree at most $\Delta \geq \Delta_0$ with a correspondence cover $\mathcal{H} = (L, H)$ and let $\epsilon \in (0, 1/2)$.
    Let $\ell, t\,:\, V(G) \to \N$, and for each $v \in V(G)$, define $\alpha_{\min}(v)$ as follows:
    \begin{align}\label{eq:alpha_min}
        \alpha_{\min}(v) \defeq \min_{\substack{S \subseteq N_G(v), \\ i(G[S]) \geq t(v)}} \overline{\alpha}(G[S]).
    \end{align}
    If for each $v \in V(G)$,
    \[|L(v)| \geq \max\left\{\frac{2\deg(v)}{(1 - \eps)^2\alpha_{\min}(v)}, \frac{2\,t(v)\,\ell(v)}{\eps}\right\},\]
    and
    \begin{enumerate}[label= \ep{\normalfont C\arabic*}, leftmargin = \leftmargin + 1\parindent]
        \item\label{item:1} $\eps(1 - \eps)\ell(v)t(v) \geq 18\log \Delta + 6\log 16$,
        \item\label{item:2} $\ell(v) \geq 36\log \Delta + 12\log 16$, and
        \item\label{item:3} $\binom{\deg(v)}{\ell(v)}/\ell(v)! < \Delta^{-3}/8$.
    \end{enumerate}
    Then $G$ admits a proper $\mathcal{H}$-coloring.
\end{theo}

Let $\eps = 1/4$, $\ell(v) \defeq \deg(v)^{1/2 + \gamma}$, and $t(v) \defeq \deg(v)^{1/2 - 2\gamma}$ for some constant $\gamma \in (0, 1/10)$ to be determined.
For each $v \in V(G)$, let $S(v) \subseteq N_G(v)$ be the minimizer of \eqref{eq:alpha_min}.
As $G$ is $(\boldk, \boldr)$-locally-sparse, it follows that $G[S(v)]$ is as well.
Furthermore, we note the following as a result of \ref{item:deg_bound_median} and \ref{item:k_r_bound_median} for $\rho$ sufficiently small:
\[|S(v)| \,\geq\, \log_2 t(v) \,\geq\, \frac{\log \deg(v)}{4} \,\geq\, \boldr(v)^{2\boldr(v)}.\]
Therefore, by Lemma~\ref{lemma:median}, we have the following:
\begin{align*}
    \alpha_{\min}(v) &\geq \frac{\log t(v)}{2\boldr(v)\log\left(\boldr(v)\tilde\boldk(v)^{\frac{1}{\boldr(v)(\boldr(v) - 1)}}\log t(v)\right)} \\
    &\geq \frac{\log \deg(v)}{8\boldr(v)\log\left(\boldr(v)\boldk(v)^{\frac{1}{\boldr(v)(\boldr(v) - 1)}}\log \deg(v)\right)} \\
    &= \frac{\log \deg(v)}{8\boldr(v)\left(\log\boldr(v) + \frac{1}{\boldr(v)(\boldr(v) - 1)}\log \boldk(v) + \log\log \deg(v)\right)} \\
    &\geq \frac{1}{(8+\mu/2)\left(\boldeps(v) + \frac{\boldr(v) \log\log\deg(v)}{\log\deg(v)}\right)},
\end{align*}
where the last step follows by the bounds on $\boldr(v)$, the definition of $\boldeps(v)$, and for $\mu$ sufficiently small.
From here, we may conclude for $\deg(v)$ large enough (which follows for $\Delta$ large enough by \ref{item:deg_bound_median}) that
\[(30+\mu)\deg(v)\left(\boldeps(v) + \frac{\boldr(v) \log\log\deg(v)}{\log\deg(v)}\right) \,\geq\, \max\left\{\frac{2\deg(v)}{(1 - \eps)^2\alpha_{\min}(v)}, \frac{2\,t(v)\,\ell(v)}{\eps}\right\}.\]
In particular, it is now enough to show that conditions \ref{item:1}--\ref{item:3} are satisfied.
Let us first consider \ref{item:1}.
We have
\begin{align*}
    \eps(1 - \eps)\ell(v) t(v) &= \frac{3\,\deg(v)^{1 - \gamma}}{16} \geq \frac{3}{16}(\log \Delta)^{2- 2\gamma} \geq 18\log \Delta + 6\log 16,
\end{align*}
where the first inequality follows by \ref{item:deg_bound_median} and the second for $\Delta$ large enough.

Similarly, we have
\[\ell(v) = \deg(v)^{1/2 + \gamma} \geq (\log \Delta)^{1 + 2\gamma} \geq 36\log \Delta + 12\log 16,\]
completing the proof of \ref{item:2}.

To prove \ref{item:3}, we will use the following inequality due to Stirling for $n$ large enough:
\[n! \geq n^{n + 1/2}e^{-n}.\]
With this in hand, we have
\begin{align*}
    \binom{\deg(v)}{\ell(v)}/\ell(v)! \,\leq\, \frac{\deg(v)^{\ell(v)}}{(\ell(v)!)^2}
    \,\leq\, \left(\frac{e^2\deg(v)}{\ell(v)^{2 + 1/\ell(v)}}\right)^{\ell(v)}
    = \left(\frac{e^2}{\deg(v)^{2\gamma + (1+2\gamma)/(2\ell(v))}}\right)^{\ell(v)}.
\end{align*}
Taking logs on both sides, it is enough to show that
\[\ell(v)\left((2\gamma + (1+2\gamma)/\ell(v))\log\deg(v) - 2\right) \geq 3\log(2\Delta).\]
As $\ell(v) > (\log\Delta)^{1+2\gamma}$, it is enough to have
\[(2\gamma + (1+2\gamma)/\ell(v))\log\deg(v) > 2,\]
which follows by \ref{item:deg_bound_median} for $\Delta$ large enough and $\gamma = 1/100$.

\subsection*{Acknowledgements}
We thank Anton Bernshteyn and Clayton Mizgerd for helpful discussions.
We also thank the anonymous referees for their helpful suggestions.
In particular, for the suggestion to employ the local occupancy framework, which greatly simplified the exposition as well as improved the results from an earlier version of this manuscript.

\vspace{0.1in}
\printbibliography

@article{bonamy2022bounding,
  title={Bounding $\chi$ by a fraction of $\Delta$ for graphs without large cliques},
  author={M. Bonamy and T. Kelly and P. Nelson and L. Postle},
  journal={Journal of Combinatorial Theory, Series B},
  volume={157},
  pages={263--282},
  year={2022},
  publisher={Elsevier}
}

@article{vizing1976coloring,
  title={Coloring the vertices of a graph in prescribed colors},
  author={V.G. Vizing},
  journal={Diskret. Analiz},
  volume={29},
  number={3},
  pages={10},
  year={1976}
}

@article{erdos1979choosability,
  title={Choosability in graphs},
  author={P. Erd{\H{o}}s and A.L. Rubin and H. Taylor},
  journal={Congr. Numer},
  volume={26},
  number={4},
  pages={125--157},
  year={1979}
}

@article{shearer1983note,
  title={A note on the independence number of triangle-free graphs},
  author={J.B. Shearer},
  journal={Discrete Mathematics},
  volume={46},
  number={1},
  pages={83--87},
  year={1983},
  publisher={Elsevier}
}

@article{shearer1995independence,
  title={On the independence number of sparse graphs},
  author={J.B. Shearer},
  journal={Random Structures \& Algorithms},
  volume={7},
  number={3},
  pages={269--271},
  year={1995},
  publisher={Wiley Online Library}
}

@article{ajtai1980note,
  title={A note on Ramsey numbers},
  author={M. Ajtai and J. Koml{\'o}s and E. Szemer{\'e}di},
  journal={Journal of Combinatorial Theory, Series A},
  volume={29},
  number={3},
  pages={354--360},
  year={1980},
  publisher={Elsevier}
}

@article{ajtai1981turan,
  title={On Tur{\'a}n’s theorem for sparse graphs},
  author={M. Ajtai and P. Erd{\H{o}}s and J. Koml{\'o}s and E. Szemer{\'e}di},
  journal={Combinatorica},
  volume={1},
  pages={313--317},
  year={1981},
  publisher={Springer}
}

@article{molloy2019list,
  title={The list chromatic number of graphs with small clique number},
  author={M. Molloy},
  journal={Journal of Combinatorial Theory, Series B},
  volume={134},
  pages={264--284},
  year={2019},
  publisher={Elsevier}
}

@article{dvovrak2018correspondence,
  title={Correspondence coloring and its application to list-coloring planar graphs without cycles of lengths 4 to 8},
  author={Z. Dvo{\v{r}}{\'a}k and L. Postle},
  journal={Journal of Combinatorial Theory, Series B},
  volume={129},
  pages={38--54},
  year={2018},
  publisher={Elsevier}
}

@article{anderson2022coloring,
  title={Coloring graphs with forbidden almost bipartite subgraphs},
  author={Anderson, James and Bernshteyn, Anton and Dhawan, Abhishek},
  journal={Random Structures \& Algorithms},
  volume={66},
  number={4},
  pages={e70012},
  year={2025},
  publisher={Wiley Online Library}
}

@article{alon1999coloring,
  title={Coloring graphs with sparse neighborhoods},
  author={N. Alon and M. Krivelevich and B. Sudakov},
  journal={Journal of Combinatorial Theory, Series B},
  volume={77},
  number={1},
  pages={73--82},
  year={1999},
  publisher={Elsevier}
}

@article{bernshteyn2019johansson,
  title={The Johansson-Molloy theorem for DP-coloring},
  author={A. Bernshteyn},
  journal={Random Structures \& Algorithms},
  volume={54},
  number={4},
  pages={653--664},
  year={2019},
  publisher={Wiley Online Library}
}

@article{anderson2023colouring,
  title={Colouring graphs with forbidden bipartite subgraphs},
  author={J. Anderson and A. Bernshteyn and A. Dhawan},
  journal={Combinatorics, Probability and Computing},
  volume={32},
  number={1},
  pages={45--67},
  year={2023},
  publisher={Cambridge University Press}
}

@unpublished{davies2020algorithmic,
  title={An algorithmic framework for colouring locally sparse graphs},
  author={E. Davies and R.J. Kang and F. Pirot and J-S. Sereni},
  howpublished = {\url{https://arxiv.org/abs/2004.07151} (preprint)},
  date={2020},
}

@unpublished{local_sparse,
  title={Coloring locally sparse graphs},
  author={J. Anderson and A. Dhawan and A. Kuchukova},
  howpublished = {\url{https://arxiv.org/abs/2402.19271} (preprint)},
  date={2024},
}

@unpublished{DKPS,
    author = {E. Davies and R.J. Kang and F. Pirot and J.-S. Sereni},
    title = {Graph structure via local occupancy},
    howpublished = {\url{https://arxiv.org/abs/2003.14361} (preprint)},
    date = {2020},
}

@unpublished{pirot2021colouring,
  title={Colouring locally sparse graphs with the first moment method},
  author={F. Pirot and E. Hurley},
  howpublished = {\url{https://arxiv.org/abs/2109.15215} (preprint)},
  date = {2021},
}

@article{frieze2006randomly,
  title={On randomly colouring locally sparse graphs},
  author={A. Frieze and J. Vera},
  journal={Discrete Mathematics \& Theoretical Computer Science},
  volume={8},
  year={2006},
  publisher={Episciences. org}
}

@article{davies2020coloring,
  title={Coloring triangle-free graphs with local list sizes},
  author={Davies, E. and de Joannis de Verclos, R. and Kang, R.J. and Pirot, F.},
  journal={Random Structures \& Algorithms},
  volume={57},
  number={3},
  pages={730--744},
  year={2020},
  publisher={Wiley Online Library}
}

@article{bonamy2020edge,
  title={Edge-colouring graphs with local list sizes},
  author={Bonamy, Marthe and Delcourt, Michelle and Lang, Richard and Postle, Luke},
  journal={Journal of Combinatorial Theory, Series B},
  volume={165},
  pages={68--96},
  year={2024},
  publisher={Elsevier}
}

@unpublished{davies2025hard,
  title={The hard-core model in graph theory},
  author={Davies, Ewan and Kang, Ross J},
  howpublished = {\url{https://arxiv.org/abs/2501.03379} (preprint)},
  year={2025}
}

@article{dhawan2023multigraph,
  title={Multigraph edge-coloring with local list sizes},
  author={Dhawan, Abhishek},
  journal={Discrete Mathematics},
  volume={348},
  number={6},
  pages={114420},
  year={2025},
  publisher={Elsevier}
}

@book{alon2016probabilistic,
  title={The probabilistic method},
  author={Alon, Noga and Spencer, Joel H},
  year={2016},
  publisher={John Wiley \& Sons}
}

@article{davies2021occupancy,
  title={Occupancy fraction, fractional colouring, and triangle fraction},
  author={Davies, Ewan and de Joannis de Verclos, R{\'e}mi and Kang, Ross J and Pirot, Fran{\c{c}}ois},
  journal={Journal of Graph Theory},
  volume={97},
  number={4},
  pages={557--568},
  year={2021},
  publisher={Wiley Online Library}
}

@article{davies2017independent,
  title={Independent sets, matchings, and occupancy fractions},
  author={Davies, Ewan and Jenssen, Matthew and Perkins, Will and Roberts, Barnaby},
  journal={Journal of the London Mathematical Society},
  volume={96},
  number={1},
  pages={47--66},
  year={2017},
  publisher={Wiley Online Library}
}

@article{davies2018average,
  title={On the average size of independent sets in triangle-free graphs},
  author={Davies, Ewan and Jenssen, Matthew and Perkins, Will and Roberts, Barnaby},
  journal={Proceedings of the American Mathematical Society},
  volume={146},
  number={1},
  pages={111--124},
  year={2018}
}

@unpublished{dhawan2024palette,
  title={Palette Sparsification for Graphs with Sparse Neighborhoods},
  author={Dhawan, Abhishek},
  howpublished = {\url{https://arxiv.org/abs/2408.08256} (preprint)},
  year={2024}
}

@BOOK{MolloyReed,
    AUTHOR = "Molloy, M. and Reed, B.",
    TITLE = "{Graph Colouring and the Probabilistic Method}",
    PUBLISHER = "Springer",
    date = "2002",
}

@article{alon1996independence,
  title={Independence numbers of locally sparse graphs and a Ramsey type problem},
  author={Alon, Noga},
  journal={Random Structures \& Algorithms},
  volume={9},
  number={3},
  pages={271--278},
  year={1996},
  publisher={Wiley Online Library}
}

\appendix

\section{Proof of Proposition~\ref{prop:embedding}}\label{appendix:prop}

For convenience, we restate the result.

\begin{prop*}[Restatement of Proposition~\ref{prop:embedding}]
    Let $G$ be a graph of maximum degree $\Delta$ and let $\boldk\,:\,V(G) \to \R$ and $\boldr\,:\,V(G) \to \N$ be such that $G$ is $(\boldk, \boldr)$-locally-sparse.
    For any $\delta \leq \Delta$, there is a graph $G'$ and functions $\tilde \boldk\,:\,V(G') \to \R$ and $\tilde \boldr\,:\,V(G') \to \N$ such that $G'$ is $(\tilde\boldk, \tilde\boldr)$-locally-sparse and the following hold for $j \defeq \max\set{\delta - \delta(G), 0}$:
    \begin{enumerate}[label= \ep{\normalfont I\arabic*}, leftmargin = \leftmargin + 1\parindent]
        \item\label{item:delta_app} $\delta(G') \geq \delta$,
        \item\label{item:Delta_app} $\Delta(G') = \Delta$,
        \item\label{item:size_app} $|V(G')| = |V(G)|\,2^{j}$, and
        \item\label{item:loc_sparsity_app} for $\ell \defeq 2^{j}$, there exist homomorphisms $\phi_1, \ldots, \phi_\ell\,:\,G\to G'$ such that 
        \begin{itemize}
            \item for each $v \in V(G)$ and $\ell' \in [\ell]$, we have $\boldk(v) = \tilde \boldk (\phi_{\ell'}(v))$ and $\boldr(v) = \tilde \boldr (\phi_{\ell'}(v))$, and
            \item for $\ell_1, \ell_2 \in [\ell]$ such that $\ell_1 \neq \ell_2$, the images $\im(\phi_{\ell_1})$ and $\im(\phi_{\ell_2})$ are vertex-disjoint.
        \end{itemize}
    \end{enumerate}
\end{prop*}

\begin{proof}
    If $\delta(G) \geq \delta$, the claim holds for $G' = G$.
    If not, we will define $G'$ iteratively as follows:
    \begin{enumerate}
        \item Let $G_0 \defeq G$.
        \item For $i \geq 0$, define $G_{i+1}$ by taking two vertex-disjoint copies of $G_i$, say $G_i^{(1)}$ and $G_i^{(2)}$.
        For each $v \in V(G_i)$ such that $\deg_{G_i}(v) < \delta$, draw an edge between the corresponding vertices in $G_i^{(1)}$ and $G_i^{(2)}$.
        \item Let $G' = G_j$, where $j$ is as defined in the statement of the proposition.
    \end{enumerate}
    Note that \ref{item:size_app} holds by construction.
    We will first show that the graphs $G_i$ satisfy $\Delta(G_i) = \Delta$ for all $0 \leq i \leq j$.
    For $i = 0$, this holds trivially.
    Suppose it holds for some $0\leq i < j$.
    Let $G_i^{(1)}$ and $G_i^{(2)}$ be the copies of $G_i$ in $G_{i+1}$.
    For $s \in \set{1, 2}$ and any vertex $v \in V(G_i^{(s)})$, $\deg_{G_{i+1}}(v) \neq \deg_{G_i^{(s)}}(v)$ if and only if $\deg_{G_i^{(s)}}(v) < \delta \leq \Delta$.
    Furthermore, by construction, we have the following:
    \begin{align}\label{eq:degree_change}
        \deg_{G_i^{(s)}}(v) < \delta \implies \deg_{G_{i+1}}(v) = \deg_{G_i^{(s)}}(v) + 1 \leq \delta.
    \end{align}
    As any vertex of degree $\Delta \geq \delta$ in $G_i$ has the same degree in $G_{i+1}$,
    it follows that $\Delta(G_{i+1}) = \Delta$.
    Setting $i = j$ completes the proof of \ref{item:Delta_app}.

    We will define functions $\boldk_i\,:\, V(G_i) \to \R$ and $\,\boldr_i\,:\,V(G_i) \to \N$ such that $G_i$ is $(\boldk_i,\boldr_i)$-locally-sparse to assist with the proof of \ref{item:loc_sparsity_app}.
    Let $\boldk_0 \defeq \boldk$ and $\boldr_0 \defeq \boldr$.
    Suppose we have defined $\boldk_i,\,\boldr_i$ for some $i \geq 0$.
    Let $G_i^{(1)}$ and $G_i^{(2)}$ be the copies of $G_i$ in $G_{i+1}$.
    For $s \in \set{1, 2}$ and any vertex $v \in V(G_i^{(s)})$, the graph $G_{i+1}[N_{G_{i+1}}(v)]$ is either isomorphic to $G_{i}^{(s)}[N_{G_{i}^{(s)}}(v)]$ or contains one additional isolated vertex.
    In particular, for $\boldk_{i+1}(v) = \boldk_i(v)$ and $\boldr_{i+1}(v) = \boldr_i(v)$, the graph $G_{i+1}$ is $(\boldk_{i+1}, \boldr_{i+1})$-locally-sparse.

    Let us now consider \ref{item:loc_sparsity_app}.
    We will prove the following more general statement which implies \ref{item:loc_sparsity_app} for $\tilde \boldk \defeq \boldk_j$ and $\tilde \boldr \defeq \boldr_j$.

    \begin{claim}
        For each $0 \leq i \leq j$, there exist $\ell_i \defeq 2^{i}$ homomorphisms $\phi_1, \ldots, \phi_{\ell_i}\,:\,G\to G_i$ such that 
        \begin{enumerate}
            \item\label{item:boldkr} for each $v \in V(G)$ and $\ell' \in [\ell_i]$, we have $\boldk(v) = \tilde \boldk_i (\phi_{\ell'}(v))$ and $\boldr(v) = \tilde \boldr_i (\phi_{\ell'}(v))$, and
            \item for $\ell_1, \ell_2 \in [\ell_i]$ such that $\ell_1 \neq \ell_2$, the images $\im(\phi_{\ell_1})$ and $\im(\phi_{\ell_2})$ are vertex-disjoint.
        \end{enumerate}
    \end{claim}

    \begin{claimproof}
        We will prove this by induction on $i$.
        For $i = 0$, the claim is trivial.
        Suppose the claim holds for some $0 \leq i < j$.
        Let $G_i^{(1)}$ and $G_i^{(2)}$ be the copies of $G_i$ in $G_{i+1}$, and for $s \in \set{1, 2}$ let $\phi_1^{(s)}, \ldots, \phi_{\ell_i}^{(s)}\, :\, G \to G_i^{(s)}$ be the homomorphisms guaranteed by the induction hypothesis.
        By definition of $\boldk_{i+1}$ and $\boldr_{i+1}$ and since the graphs $G_i^{(1)}$ and $G_i^{(2)}$ are vertex-disjoint, it follows that the homomorphisms 
        \[\set{\phi_\ell^{(s)}\,:\, s \in \set{1, 2},\, 1 \leq \ell \leq \ell_i},\]
        satisfy the conditions of the claim.
        As $\ell_{i+1} = 2\ell_i$, this completes the proof.
    \end{claimproof}

    Finally, note that by \eqref{eq:degree_change}, we may conclude the following for $i \geq 0$:
    \[\delta(G_i) < \delta \implies \delta(G_{i+1}) = \delta(G_i) + 1.\]
    In particular, \ref{item:delta_app} holds for $G'$ by definition of $j$, completing the proof.
\end{proof}

\end{document}